\documentclass[a4paper]{article}
\usepackage[utf8]{inputenc}
\usepackage{amsmath,bm}
\usepackage{amssymb}
\usepackage{amsthm}
\usepackage{mathtools}
\usepackage{cases}
\usepackage{graphicx,color}
\usepackage{float}
\usepackage{here}
\usepackage{tikz}
\usepackage{comment}
\usepackage{ulem}
\usepackage[T1]{fontenc}
\usepackage{textcomp}

\theoremstyle{plain}
\newtheorem{thm}{Theorem}[section]
\newtheorem{lem}[thm]{Lemma}
\newtheorem{cor}[thm]{Corollary}
\newtheorem{prop}[thm]{Proposition}

\theoremstyle{definition}
\newtheorem{df}[thm]{Definition}
\newtheorem{eg}[thm]{Example}

\theoremstyle{remark}
\newtheorem{rem}[thm]{Remark}

\newcommand{\relmiddle}[1]{\mathrel{}\middle#1\mathrel{}}
\numberwithin{equation}{section}
\allowdisplaybreaks[3]

\newcommand{\E}{\mathbb{E}}
\newcommand{\N}{\mathbb{N}}
\renewcommand{\P}{\mathbb{P}}
\newcommand{\Z}{\mathbb{Z}}

\newcommand{\cP}{\mathcal{P}}

\renewcommand{\a}{\alpha}
\newcommand{\dl}{\delta}
\newcommand{\eps}{\varepsilon}

\newcommand{\om}{\omega}

\newcommand{\del}{\partial}

\DeclareMathOperator{\gir}{gir}
\DeclareMathOperator{\im}{Im}
\DeclareMathOperator{\MC}{MC}
\DeclareMathOperator{\MH}{MH}

\DeclareMathOperator{\rk}{rk}
\DeclareMathOperator{\var}{Var}

\usepackage{xcolor}

\newcommand{\Blue}{\textcolor{blue}}
\newcommand{\Red}{\textcolor{red}}

\title{
Girth, magnitude homology, and phase transition of diagonality
}

\author{Yasuhiko \textsc{Asao}\thanks{Center for Advanced Intelligence Project, RIKEN. \texttt{yasuhiko.asao@riken.jp}} \and Yasuaki \textsc{Hiraoka}\thanks{Kyoto University Institute for Advanced Study, WPI-ASHBi, Kyoto University. Center for Advanced Intelligence Project, RIKEN. \texttt{hiraoka.yasuaki.6z@kyoto-u.ac.jp}} \and  Shu \textsc{Kanazawa}\thanks{Kyoto University Institute for Advanced Study, Kyoto University. \texttt{kanazawa.shu.73m@st.kyoto-u.ac.jp}}}
\date{\today}

\begin{document}
\maketitle
\begin{abstract}
This paper studies the magnitude homology of graphs focusing mainly on the relationship between its diagonality and the girth.
Magnitude and magnitude homology are formulations of the Euler characteristic and the corresponding homology, respectively, for finite metric spaces, first introduced by Leinster and Hepworth--Willerton. Several authors study them restricting to graphs with path metric, and some properties which are similar to the ordinary homology theory have come to light. However, the whole picture of their behavior is still unrevealed, and it is expected that they catch some geometric properties of graphs. In this article, we show that the girth of graphs partially determines magnitude homology, that is, the larger girth a graph has, the more homologies near the diagonal part vanish. Furthermore, applying this result to a typical random graph, we investigate how the diagonality of graphs varies statistically as the edge density increases. In particular, we show that there exists a phase transition phenomenon for the diagonality. 
\end{abstract}
\section{Introduction}
The {\it magnitude} of finite metric spaces was introduced by Leinster~\cite{L13} as a formulation of Euler characteristic of finite metric spaces. Magnitude has several interesting properties such as multiplicativity property and inclusion-exclusion principle, which seems parallel to the case of ordinary Euler characteristic of topological spaces. However, whole picture of the behavior of magnitude is unrevealed, and that is attracting people in several areas of mathematics. In particular, magnitude of finite graphs, which takes values in formal power series with $\mathbb{Z}$-coefficients, is studied by several authors so far (\cite{AI20},~\cite{BK},~\cite{Gu},~\cite{HW17},~\cite{L19}).
Throughout this article, we call a finite, simple, and undirected graph without loops just a graph.

The {\it magnitude homology} of graphs is a categorification of magnitude, first introduced by Hepworth--Willerton~\cite{HW17} as an analogy of ordinary homology theory. It is a bigraded abelian group whose Euler characteristic coincides with the magnitude, and the multiplicativity property and the inclusion-exclusion principle are formulated as the K\"unneth and the Mayer--Vietoris theorems, respectively~\cite{HW17}. Their beautiful theory enables us to compute the magnitude and magnitude homology of graphs. For example, Gu~\cite{Gu} showed a remarkable compatibility of magnitude homology with algebraic Morse theory, and he computed magnitude homology of several types of graphs including well-known classical ones. Bottinelli--Kaiser~\cite{BK} study the magnitude homology of median graphs, using the retraction between homology groups. More or less, the remarkable property concerned in their works is the {\it diagonality} of graphs, first suggested in~\cite{HW17}, which guarantees simpleness of the magnitude homology in some sense.

In this article, we show that the girth of graphs partially determines magnitude homology, that is, the larger girth a graph has, the more homologies near the diagonal part vanish. Furthermore, by using this result, we investigate how the diagonality of graphs varies statistically as the edge density (proportion of the number of edges to that of possible edges) increases. In particular, we show that there exists a phase transition phenomenon for the diagonality.
As shown in~\cite{HW17}, a tree (or more generally, a forest) which has low edge density is diagonal.
It is also known that a few graphs with high edge density are diagonal. This fact is shown in~\cite{HW17} for complete graph, and in~\cite{Gu} for {\it pawful graph}~(see Definition~\ref{df:pawful}).
However, graphs with intermediate edge density are more likely to be non-diagonal.
To describe this phenomenon statistically, we consider the Erd\H os--R\'enyi graph model which is a typical random graph model extensively studied since the 1960s~(\cite{ER59},~\cite{ER60},~\cite{Gi59}). Given $n\in\N$ and $p\in[0,1]$, an Erd\H os--R\'enyi graph $G_{n,p}$ with parameters $n$ and $p$ is a random graph with $n$ vertices, where the edge between each pair of vertices is added independently with probability $p$.

Now, we explain our results in the following. We first state a relationship between girth of graphs and magnitude homology. They will be proved in an algebraic and combinatorial way in Section \ref{algebraicpart}. Let $G$ be a graph and $x \in V(G)$ be a vertex. We define the {\it local girth of} $G$ {\it at} $x$ by
\[
\gir_x(G)\coloneqq\inf\{i\ge3\mid\text{ there exists a cycle of length $i$ in $G$ containing $x$}\}.
\]
We also define the {\it girth  of} $G$ by 
\[
\gir(G) \coloneqq \min_{x} \gir_{x}(G).
\]
Note that the following statements are compatible with the computation of magnitude homology for trees and cycle graphs in~\cite{Gu} and~\cite{HW17}, respectively. In particular, Corollary~\ref{cor:girthDiag} is a generalization of the computation of magnitude homology of trees in \cite[Corollary~6.8]{HW17}.
Below, $\MH_{*,*}(G)$ is the magnitude homology of $G$, and the superscript $x$ of $\MH_{*,*}^x(G)$ indicates the restriction on the starting point (see Section~\ref{ssec:MH} for the definitions).
\begin{thm}\label{lpart}
Let $\ell \geq 1$. If $\gir_{x}(G)\geq 5$, then
\[
\MH^{x}_{\ell, \ell}(G) \cong \mathbb{Z}^{{\rm deg}x},
\]
where ${\rm deg}x$ denotes the degree of the vertex $x$. 
\end{thm}
The following is also obtained by Sazdanovic--Summers in \cite[Thoerem~4.3]{RV}.
\begin{cor}
Let $\ell \geq 1$. If $\gir(G)\geq 5$, then
\[
\MH_{\ell, \ell}(G) \cong \mathbb{Z}^{2\#E(G)},
\]
where $\#E(G)$ denotes the number of edges of $G$.
\end{cor}
The following are extensions of the above.
\begin{thm}\label{otherpart}
Let $\ell \geq 1$ and $i\ge0$. If $\gir_{x}(G)\geq  2i + 5$, then 
\[
\MH^{x}_{\ell -j, \ell}(G) \cong\begin{cases}
\mathbb{Z}^{{\rm deg}x},    & j = 0,\\
0,                          & 1 \leq j \leq i.
\end{cases}
\]
\end{thm}
\begin{cor}\label{cor:girthDiag}
Let $\ell\geq1$ and $i\ge0$. If $\gir(G)\geq  2i + 5$, then 
\[
\MH_{\ell -j, \ell}(G) \cong\begin{cases}
\mathbb{Z}^{2\# E(G)},  & j = 0,\\
0,                      & 1 \leq j \leq i.
\end{cases}
\]
\end{cor}
The above results will be proved by using algebraic Morse theory. The following gives a criterion for the diagonality of graphs. Let $e \in E(G)$ be an edge. We define the {\it local girth of} $G$ {\it at} $e$ by 
\[
\gir_e(G)\coloneqq\inf\{i\ge3\mid\text{ there exists a cycle of length $i$ in $G$ containing $e$ as its edge}\}.
\]
Note that we have $\gir(G) = \min_{e} \gir_{e}(G)$.
\begin{thm}\label{nondiag}
Let $G$ be a graph and $e \in E(G)$ be an edge. If $k\coloneqq\gir_e(G)\in[5,\infty)$, then $\MH_{2, \ell}(G) \neq 0$ for $\ell = \lfloor \frac{k+1}{2} \rfloor$.
\end{thm}
\begin{cor}
If $G$ is a diagonal graph, then $\gir(G)=3,4$, or $\infty$.
\end{cor}
By considering $k = 2i + 5$ or $2i+6$ in Theorem \ref{nondiag}, it turns out that the range $1 \leq j \leq i$ guaranteeing the vanishing of magnitude homology groups in Corollary \ref{cor:girthDiag} is optimal.

Next we state stochastic properties of magnitude homology with respect to the Erd\H os--R\'enyi random graph model.
They will be shown in Section \ref{stochasticpart}. In the study of the Erd\H os--R\'enyi graph $G_{n,p}$, one is usually concerned with the asymptotic behavior of $G_{n,p}$ as the number of vertices $n$ tends to infinity, where $p$ is typically regarded as a function of $n$.
For a graph property $\cP$, we say that $G_{n,p}$ satisfies $\cP$ asymptotically almost surely (a.a.s.) if $\lim_{n\to\infty}\P(G_{n,p}\text{ satisfies }\cP)=1$.
We also use the Bachmann--Landau big-$O$/little-$o$ notation with respect to the number of vertices $n$ tending to infinity. Additionally, for non-negative functions $f(n)$ and $g(n)$, $f(n)=\om(g(n))$ mean that $g(n)=o(f(n))$.
One of the most classical themes is searching the threshold probability $p(n)$ for various graph properties $\cP$. Here, we call the probability $p(n)$ a threshold for $\cP$ if $p=o(p(n))$ implies that $G_{n,p}$ satisfies $\cP$ a.a.s. and $p=\om(p(n))$ implies that $G_{n,p}$ does not satisfy $\cP$ a.a.s. For example, $p(n)=n^{-1}$ is the threshold probability for the appearance of a cycle in $G_{n,p}$. 

The first result exhibits a phase transition for the diagonality of Erd\H os--R\'enyi graphs. This is where the magnitude homology of Erd\H os--R\'enyi graph suddenly becomes non-diagonal.
\begin{thm}\label{thm:main1}
Let $G_{n,p}$ be an Erd\H os--R\'enyi graph with parameters $n$ and $p$. Then, the following $(1)$, $(2)$, and $(3)$ hold.
\begin{enumerate}
\item If $p=o(n^{-1})$, then $G_{n,p}$ is diagonal a.a.s.
\item If $p=cn^{-1}$, then
\[
\lim_{n\to\infty}\P(G_{n,p}\text{ is non-diagonal})=\begin{cases}
1-\sqrt{1-c}\exp(c/2+c^2/4+c^3/6+c^4/8), &0<c<1,\\
1,                                       &c>1.
\end{cases}
\]
\item If $p=\om(n^{-1})$ and $p=o(n^{-3/4})$, then $G_{n,p}$ is non-diagonal a.a.s.
\end{enumerate}
\end{thm}
\begin{figure}[H]
\centering
\includegraphics[width=10cm,bb=0 0 672 258]{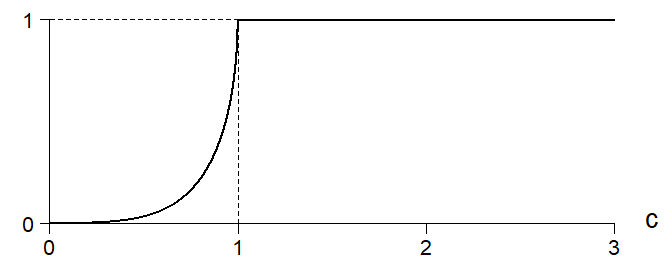}\\
\caption{The limiting function of $c$ appearing in Theorem~\ref{thm:main1}~(2).}
\label{fig:threDist}
\end{figure}
As seen in Figure~\ref{fig:threDist}, the probability that $G_{n,c/n}$ is non-diagonal approaches an explicit constant bounded away from one whenever $c<1$. Meanwhile, when $c>1$, $G_{n,c/n}$ is non-diagonal a.a.s.

A graph property $\cP$ is said to be monotone increasing if whenever a graph $G$ satisfies $\cP$ and $G$ is a subgraph of a graph $G'$ then $G'$ also satisfies $\cP$. Every monotone property has a threshold probability in Erd\H os--R\'enyi graphs~\cite{BT87}. However, since non-diagonality is not a monotone increasing graph property, it is natural to seek what happens in the regime of $p$ that Theorem~\ref{thm:main1} does not cover. The following theorem partially answers this question. 
\begin{thm}\label{thm:main2}
Let $\eps>0$ be fixed, and let $G_{n,p}$ be an Erd\H os--R\'enyi graph with parameters $n$ and $p$. Then,
\[
p\ge\Biggl(\frac{(3+\eps)\log n}n\Biggr)^{1/3}
\]
implies that $G_{n,p}$ is diagonal a.a.s.
\end{thm}
The behavior of the probability that $G_{n,p}$ is non-diagonal in the regime of $p$ that both Theorems~\ref{thm:main1} and~\ref{thm:main2} do not cover should be studied as a further theme. At this moment, even the existence of the threshold where $G_{n,p}$ again becomes diagonal is still unknown.

Finally, we show the asymptotic behavior of each rank of magnitude homology around the threshold probability. The following result can be regarded as a weak law of large numbers for the rank of magnitude homology.
\begin{thm}\label{thm:main3}
Let $k,\ell\in\N$ and $p=cn^{-1}$ for some fixed $c>0$. Let $G_{n,p}$ be an Erd\H os--R\'enyi graph with parameters $n$ and $p$. Then,
\[
\lim_{n\to\infty}\frac{\E[\rk(\MH_{k,\ell}(G_{n,p}))]}n=c\dl_{k,\ell},
\]
where $\dl_{k,\ell}$ is the Kronecker delta function.
Moreover, for any $\eps>0$,
\[
\lim_{n\to\infty}\P\biggl(\biggl|\frac{\rk(\MH_{k,\ell}(G_{n,p}))}n-c\dl_{k,\ell}\biggr|>\eps\biggr)=0.
\]
\end{thm}
\begin{rem}
Theorem~\ref{thm:main3} immediately implies that for any vertex $x$ in $G_{n,p}$,
\[
\lim_{n\to\infty}\E[\rk(\MH_{k,\ell}^x(G_{n,c/n}))]=c\dl_{k,\ell}.
\]
Note that the value $c$ appearing here coincides with the limit of the expected degree of $x$ in $G_{n,c/n}$. This means that $\E[\rk(\MH_{k,\ell}^x(G_{n,p}))]$ and $\E[(\deg x)\dl_{k,\ell}]$ are asymptotically equal.
On the other hand, it is shown in~\cite{HW17} that $\rk(\MH_{k,\ell}^x(T))=(\deg x)\dl_{k, \ell}$ for any tree $T$ and its vertex $x$. Therefore, $\E[\rk(\MH_{k,\ell}^x(G_{n,p}))]$ and $\rk(\MH_{k,\ell}^x(T))$ depend only on the degree of $x$ asymptotically.
This property is compatible with the fact that $G_{n,c/n}$ has locally tree-like structure.
\end{rem}

The magnitude $\#G(q)$ of a graph $G$, which takes value in the formal power seriese $\Z[\![q]\!]$, is determined by the magnitude homology of $G$ (cf.~\cite[Theorem~2.8]{HW17}):
\[
\#G(q)=\sum_{\ell=0}^\infty\Biggl(\sum_{k=0}^\ell(-1)^k\rk(\MH_{k,\ell}(G))\Biggr)q^\ell.  
\]
For $\ell\ge0$, define $\chi_\ell(G)$ as the coefficient of $q^\ell$ in the above equation. Then, the following corollary of Theorem~\ref{thm:main3} immediately follows.
\begin{cor}
Let $\ell\in\N$ and $p=cn^{-1}$ for some fixed $c>0$. Let $G_{n,p}$ be an Erd\H os--R\'enyi graph with parameters $n$ and $p$. Then,
\[
\lim_{n\to\infty}\frac{\E[\chi_\ell(G_{n,p})]}n=(-1)^\ell c.
\]
Moreover, for any $\eps>0$,
\[
\lim_{n\to\infty}\P\biggl(\biggl|\frac{\chi_\ell(G_{n,p})}n-(-1)^\ell c\biggr|>\eps\biggr)=0.
\]
\end{cor}

This article is organized as follows. In Section 2, we briefly review some basic definitions of the magnitude homology of graphs. In Section 3, we study the magnitude homology of graphs and its diagonality from a viewpoint of girth.
We use algebraic Morse theory and combinatorial arguments on graphs. Finally, in Section 4, we study the magnitude homology of Erd\H os--R\'enyi graphs using theorems obtained in Section~3 together with classical results on random graphs.

\subsection*{Acknowledgement}
The first author is supported by RIKEN Center for Advanced Intelligence Project (AIP). The second author is supported by JST CREST Mathematics (15656429), JSPS Grant-in-Aid for Scientific Research (A) (20221963), and JSPS Grant-in-Aid for Challenging Research 
(Exploratory) (19091210).
The third author is supported by JSPS KAKENHI Grant Number 19J11237.
\section{Notations for magnitude homology of graphs}
In this section, we recall some definitions of the magnitude homology of graphs.
\subsection{Graph}
A {\it finite simple undirected graph without loops} is a pair of a nonempty finite set $V$ and a collection $E$ of subsets in $V$ of cardinality two. We regard $V$ and $E$ as a vertex set and an edge set, respectively. Throughout this article, we call a finite simple undirected graph without loops just a graph.
Below, we describe some notation and terminology for a given graph $G=(V(G),E(G))$.
\begin{df}
We say that $x\in V(G)$ is {\it adjacent to} $y\in V(G)$ if $\{x,y\}\in E(G)$, and denote $x\sim y$. For $x\in V(G)$, the {\it degree} $\deg x$ indicates the number of vertices that are adjacent to $x$. 
\end{df}
\begin{df}
A tuple $(x_0,x_1,\ldots,x_k)\in V(G)^{k+1}$ is called a {\it path between} $x,y\in V(G)$ if $x_0=x$, $x_k=y$, and $x_{i-1}\sim x_i$ for all $i=1,2,\ldots,k$.
A graph $G$ is said to be {\it connected} if for any two vertices $x,y\in V(G)$, there exists a path between $x$ and $y$.
\end{df}
\begin{df}
Let $i \geq 3$. An $i$-{\it cycle} or {\it cycle} in a graph $G$ is a tuple $(x_{0}, \dots, x_{i})$ of vertices in $G$ satisfying
\begin{itemize}
\item $\{x_{k}, x_{k+1}\} \in E(G)$  for $0 \leq k \leq i-1$,
\item $x_{0} = x_{i}$,
\item $x_{0}, \dots, x_{i-1}$ are all distinct.
\end{itemize}
\end{df}
\begin{df}
A {\it tree} is a connected graph that has no cycles, while a connected graph that has exactly one cycle is called a {\it unicyclic graph}.
\end{df}
For vertices $x,y\in V(G)$, an extended metric $d(x,y)$ is defined as the length of shortest path between $x$ and $y$, and if there exist no such paths, we set $d(x,y)=\infty$.

\subsection{Magnitude homology}\label{ssec:MH}
Let $G=(V(G),E(G))$ be a graph. For a tuple $(x_0,x_1,\ldots,x_k)\in V(G)^{k+1}$, we define
\[
L(x_0,x_1,\ldots,x_k)\coloneqq\sum_{i=1}^k d(x_{i-1},x_i).
\]
Let $\ell\in\Z_{\ge0}$ be fixed, and for any $k\in\Z_{\ge0}$, we define a free $\mathbb{Z}$-module $\MC_{k,\ell}(G)$ generated by a set
\[
\{(x_0,x_1,\ldots,x_k)\in V(G)^{k+1}\mid x_0\neq x_1\neq\cdots\neq x_k,L(x_0,\ldots,x_k)=\ell\}.
\]
We note from the definition that $\MC_{k,\ell}(G)=0$ for $k>\ell$. We can decompose $\MC_{k,\ell}(G)$ into spatially localized versions as follows. For any $k\in\Z_{\ge0}$ and $x, y \in V(G)$, we define free $\mathbb{Z}$-modules $\MC_{k,\ell}^x(G)$ and $\MC_{k,\ell}^{x, y}(G)$ generated by  sets
\begin{align*}
\{(x_0,x_1,\ldots,x_k)\in V(G)^{k+1}\mid x=x_0\neq x_1\neq\cdots\neq x_k,L(x_0,\ldots,x_k)=\ell\},
\shortintertext{and}
\{(x_0,x_1,\ldots,x_k)\in V(G)^{k+1}\mid x=x_0\neq x_1\neq\cdots\neq x_k = y,L(x_0,\ldots,x_k)=\ell\},
\end{align*}
respectively. Then we have obvious decompositions
\begin{equation}\label{chaindecomp}
\MC_{k,\ell}(G) \cong \bigoplus_{x\in V(G)}\MC_{k,\ell}^x(G) \cong  \bigoplus_{x, y\in V(G)}\MC_{k,\ell}^{x, y}(G).
\end{equation}
\begin{df}
Given
\[
(x_0,\ldots,x_i,\ldots,x_k)\in \MC_{k,\ell}(G),
\]
we say that $x_i$ {\it is a smooth point of} $(x_0,\ldots,x_i,\ldots,x_k)$ if $L(x_0,\ldots,x_k)=L(x_0,\ldots,\hat x_i,\ldots,x_k)$, that is, 
 \[
 d(x_{i-1},x_{i+1})=d(x_{i-1},x_i)+d(x_i,x_{i+1}).
 \]
 Here, the hat symbol over $x_i$ indicates that this vertex is deleted from $(x_0,\ldots,x_i,\ldots,x_k)$. We say that $x_i$ {\it is a singular point of} $(x_0,\ldots,x_i,\ldots,x_k)$ if it is not a smooth point of $(x_0,\ldots,x_i,\ldots,x_k)$.
\end{df}
For $k\ge1$, the boundary map $\del_{k,\ell}(G)\colon\MC_{k,\ell}(G)\to\MC_{k-1,\ell}(G)$ is defined as the linear extension of
\[
\del_{k,\ell}(G)(x_0,\ldots,x_k)=\sum_{i=1}^{k-1}(-1)^i1_{\{x_i\text{ is smooth}\}}(x_0,\ldots,\hat x_i,\ldots,x_k)
\]
for $(x_0,\ldots,x_k)\in \MC_{k,\ell}(G)$. By convention, we also define $\MC_{-1,l}(G)=0$ and $\del_{0,l}(G)=0$. Then, it holds that $\del_{k,\ell}(G)\circ \del_{k+1,\ell}(G)=0$ for $k\ge0$, that is, $\ker\del_{k,\ell}(G)\supset\im\del_{k+1,\ell}(G)$.
The magnitude homology group $\MH_{k,\ell}(G)$ of length $\ell$ is defined by $\MH_{k,\ell}(G)\coloneqq\ker\del_{k,\ell}(G)/\im\del_{k,\ell}(G)$.

Obviously, the boundary maps are compatible with the decompositions~\eqref{chaindecomp}. Hence it induces the decompositions
\begin{equation}\label{eq:decomp}
\MH_{k,\ell}(G) \cong \bigoplus_{x\in V(G)}\MH_{k,\ell}^x(G) \cong  \bigoplus_{x, y\in V(G)}\MH_{k,\ell}^{x, y}(G).
\end{equation}
Note that, if $x$ and $y$ are adjacent, we have a tuple $(x, y, x, \dots)$ which is a homology cycle in $\MH_{\ell,\ell}^{x}(G)$.
Hence we have $\rk (\MH_{\ell, \ell}^{x}(G)) \geq \deg x$. In particular, $\rk (\MH_{\ell, \ell}(G)) \geq 2\# E(G)$ holds from Eq.~\eqref{eq:decomp}.
\begin{eg}[{\cite[Corollary~6.8]{HW17}}]\label{eg:tree}
Let $T$ be a tree, and $x\in V(T)$ be fixed. Then we have
\[
\MH_{k,\ell}^x(T)\simeq\begin{cases}
\Z, 				&k=\ell=0,\\
\Z^{\deg x}, 		&k=\ell\ge1,\\
0,	        		&k\neq\ell.
\end{cases}
\]
This is verified by using Mayer--Vietoris Theorem in~\cite[Theorem~6.6]{HW17} after checking that it is compatible with the decompositions~\eqref{eq:decomp}. Moreover, Eq.~\eqref{eq:decomp} yields
\[
\MH_{k,\ell}(T)\simeq\begin{cases}
\Z^{\#V(T)}, 		&k=\ell=0,\\
\Z^{2\#E(T)}, 		&k=\ell\ge1,\\
0,			    	&k\neq\ell.
\end{cases}
\]
\end{eg}
\begin{df}[{\cite[Definition~7.1]{HW17}}]
A graph $G$ is called {\it diagonal} if $\MH_{k, \ell}(G) = 0$ for $k\neq \ell$.
\end{df}
\begin{df}[{\cite[Definition~4.2]{Gu}}]\label{df:pawful}
A graph of diameter at most two is called {\it pawful} if any distinct vertices $x,y,z\in V(G)$ with $d(x,y)=d(y,z)=2$ and $d(z,x)=1$  have a common neighbor. Here, for $S\subset V(G)$, a vertex $w\in V(G)$ is said to be a common neighbor of $S$ if $w$ is adjacent to all the vertices in $S$.
\end{df}
\begin{eg}
Trees are diagonal, as seen in Example~\ref{eg:tree}. Join graphs, in particular complete graphs, are also diagonal~\cite[Theorem~7.5]{HW17}. Moreover, pawful graphs are diagonal~\cite[Theorem~4.4]{Gu}.
\end{eg}

\section{Girth and magnitude homology of graphs}\label{algebraicpart}
In this section, we study algebraically the magnitude homology of graphs. First in Section \ref{alg1}, we briefly review algebraic Morse theory, which is a crucial tool for the latter parts. In Sections~\ref{alg2} and~\ref{alg3}, we compute the $(\ell - i, \ell)$-part $\MH_{\ell - i, \ell}(G)$ of magnitude homology for a general graph $G$ and for some $0 \leq i \leq \ell -1$. In Section~\ref{alg4}, we give a criterion for graphs to be diagonal. All the main results proved in this section, especially Theorems~\ref{otherpart} and~\ref{nondiag}, will be key lemmas for the probabilistic study of magnitude homology in Section \ref{sec:proofs}.

\subsection{Algebraic Morse Theory}\label{alg1}
For our computation, we use algebraic Morse theory studied in~\cite{Sk}. The matching that we construct is quite similar to that of Gu's (\cite{Gu}), while he constructs matchings for several special graphs in~\cite{Gu}.  In this subsection, we briefly review the algebraic Morse theory. It is almost the same instruction as in~\cite{Gu}, and see~\cite{Sk} for the detail.

Let $C_{\ast} = (C_{\ast}, \partial_{\ast})$ be a chain complex of finite rank free $\mathbb{Z}$-modules. We set
\[
C_{k} = \bigoplus_{\alpha \in I_{k}} C_{k, \alpha} \cong \bigoplus_{\alpha \in I_{k}}\mathbb{Z}
\]
for each $k \geq 0$. We denote differentials restricted to each component as
\[
f_{\beta \alpha}\colon C_{k+1, \alpha} \hookrightarrow C_{k+1} \xrightarrow{\partial_{k+1}}  C_{k} \twoheadrightarrow C_{k, \beta}.
\]
Let $\Gamma_{C_{\ast}}$ be a directed graph whose vertex set is $\coprod_{k} I_{k}$, and directed edges are $\{\alpha \to \beta \mid f_{\beta \alpha} \neq 0\}$. Recall that a {\it matching}
of a directed graph is a subset $M$ of the edge set such that any two distinct edges in $M$ have no common vertices. For a matching $M$ of $\Gamma_{C_{\ast}}$, we define a new directed graph $\Gamma_{C_{\ast}}^{M}$ by inverting the direction of all edges in $M$.
\begin{df}\label{morsematching}
The matching $M$ is called {\it Morse matching} if the directed graph $\Gamma_{C_{\ast}}^{M}$ is acyclic, and all homomorphisms of the form
\[
f_{\beta \alpha}\colon C_{k+1, \alpha} \hookrightarrow C_{k+1} \xrightarrow{\partial_{k+1}}  C_{k} \twoheadrightarrow C_{k, \beta}
\]
corresponding to the edges in $M$ are isomorphisms. 
 \end{df}
 Here we remark that $\Gamma_{C_{\ast}}^{M}$ is acyclic if and only if there are no closed paths in $\Gamma_{C_{\ast}}^{M}$ of the form
\[
 a_{1}  \longrightarrow b_{1} \longrightarrow \dots \longrightarrow b_{p-1} \longrightarrow a_{p}=a_1
\]
with $a_{i} \in C_{k+1}$ and $b_{i}\in C_{k}$ for some $k$.
 \begin{thm}[\cite{Sk}]\label{morsefact}
 For a Morse matching $M$, the chain complex $C_{\ast}$ is homotopy equivalent to the chain complex $\mathring C_{\ast}$ defined as follows$:$ Let $\mathring{I}_{k}$ be a subset of $I_{k}$ which consists of vertices contained in no edges in $M$. We define 
 \[
 \mathring{C}_{k} = \bigoplus_{\alpha \in \mathring{I}_{k}} C_{k, \alpha}
 \]
 for each $k \geq 0$. For each $\alpha \in \mathring{I}_{k}$ and $\beta \in \mathring{I}_{k-1}$, let $\Gamma_{\alpha, \beta}^{M}$ be the set of paths in $\Gamma_{C_{\ast}}^{M}$ connecting $\alpha$ and $\beta$ in this order. For $\gamma \in \Gamma_{\alpha, \beta}^{M}$, we define $\mathring{\partial}_{k}\colon C_{k, \alpha} \to C_{k-1, \beta}$ as
 \[
 \mathring{\partial}_{\gamma} = (-1)^{i/2}f_{\beta v_{i}}\circ f_{v_{i} v_{i-1}}^{-1} \circ \dots \circ f_{v_{2} v_{1}}^{-1}\circ f_{v_{1}\alpha},
 \]
where $\gamma = (\alpha \to v_{1} \to \dots \to v_{i} \to \beta)$.  Then the differential $\mathring{\partial}_{k}$ restricted on $C_{k, \alpha}$ for $\alpha \in \mathring{I}_{k}$ is defined as 
\[
\mathring{\partial}_{k}|_{C_{k, \alpha}} = \sum_{\beta \in \mathring{I}_{k-1}, \gamma \in \Gamma_{\alpha, \beta}^{M}} \mathring{\partial}_{\gamma}.
\]
In particular, we have $\mathring{\partial}_{k} = 0$ if the original differential $\partial_{k}$ vanishes on $\mathring{I}_{k}$.
 \end{thm}
 \subsection{Computation for diagonal part}\label{alg2}
 In this subsection, we study the diagonal part ($(\ell, \ell)$-part) of magnitude homology. In the following, we assume that $\ell \geq 1$ unless otherwise noted.
 We first recall the definition of the local girth of a graph at a fixed vertex, as seen in the introduction.
\begin{df}Let $G$ be a graph and $x \in V(G)$ be a vertex. We define the {\it local girth of} $G$ {\it at} $x$ by 
\begin{align*}
\gir_x(G)\coloneqq\inf\{i\ge3\mid\text{ there exists an $i$-cycle in $G$ containing $x$}\}.
\end{align*}
We also deine the {\it girth  of} $G$ by 
\[
\gir(G) \coloneqq \min_{x} \gir_{x}(G)
\]
\end{df}
Our subject in this subsection is to prove Theorem~\ref{lpart}. We use the algebraic Morse theory for the proof. Let us consider a  truncated chain complex
\[
0 \longrightarrow \MC^{x}_{\ell, \ell}(G) \longrightarrow \MC^{x}_{\ell -1, \ell}(G) \longrightarrow 0
\]
and denote it by $C_{\ast}$. It is easy to see that the first homology of $C_{\ast}$ is isomorphic to $\MH^{x}_{\ell, \ell}(G)$. 
For graphs that  have neither $3$- nor $4$-cycles containing $x$ as their vertex, we give a Morse matching to $C_{\ast}$.
In the following, we give a Morse matching to $C_\ast$ with $\gir_{x}(G)\geq 5$.
\begin{lem}\label{backandforth}
Let $\ell \geq 1$ and $i \geq 1$. Let $G$ be a graph with $\gir_{x}(G)\geq 4$ for a vertex $x \in V(G)$.
Let
\[
(x = x_{0}, \dots, x_{\ell}) \in \MC^{x}_{\ell, \ell}(G)
\]
be a chain, and suppose that $x_{j}$ is its singular point for $0 \leq j \leq i-1$. Then $x_{i} \in \{x_{0}, x_{1}\}$.
\end{lem}
\begin{proof}
We prove by induction on $i$. For $i = 1$, the statement is trivially true. Suppose that $x_{j}$ is singular for $0 \leq j \leq i-1$ and $x_{j} \in \{x_{0}, x_{1}\}$ for $0 \leq j \leq i-1$. Then we have $\{ x_{i-2}, x_{i-1}\} = \{x_{0}, x_{1}\}$ because $x_{i-2} \neq  x_{i-1}$. Note here that we have $d(x_{k}, x_{k+1}) =1$ for $0 \leq k \leq \ell-1$ by the definition of $\MC_{\ell, \ell}(G)$. Then by the assumption that $x_{i-1}$ is a singular point, we have $d(x_{i-2}, x_{i}) \leq 1$. If we have $d(x_{i-2}, x_{i}) = 1$, then these three points $x_{i-2}, x_{i-1}, x_{i}$ form a $3$-cycle containing $x$ because $x_{i-2}$ or $x_{i-1}$ coincides with $x$, which is not the case (see Figure~\ref{triangle}). Hence we obtain that $d(x_{i-2}, x_{i}) = 0$, which implies that $x_{i} = x_{i-2} \in \{ x_{0}, x_{1}\}$.
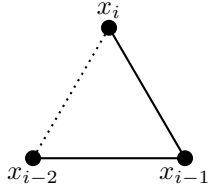
\begin{figure}[H]
\centering
\begin{tikzpicture}
\coordinate (a) at (0, 0) node [below] at (a) {$x_{i-2}$};
\coordinate (b) at (2, 0) node [below] at (b) {$x_{i-1}$};
\coordinate (c) at (1, 1.73) node [above] at (c) {$x_{i}$};
\fill (a) circle (3pt);
\fill (b) circle (3pt);
\fill (c) circle (3pt);
\draw[thick] (a) -- (b);
\draw[thick] (b) -- (c);
\draw[thick][dotted] (a) -- (c);
\end{tikzpicture}
\caption{An illustration of a $3$-cycle containing $x$ in the case that $d(x_{i-2},x_i) = 1$.}
\label{triangle}
\end{figure}
\end{proof}
Let $T_{\ell}$ be a subset of generators in $\MC^{x}_{\ell, \ell}(G)$ defined as
\[
T_{\ell} = \left\{ (x_{0}, \dots, x_{\ell}) \in \MC^{x}_{\ell, \ell}(G) \mid\text{some $x_{i}$'s are smooth for $0\leq i \leq \ell$}\right\}.
\]
Whenever $T_\ell\neq\emptyset$, we define a map $f_{\ell}\colon T_{\ell} \longrightarrow \MC^{x}_{\ell-1, \ell}(G)$ by deleting the first smooth point, that is,
\[
f_{\ell}(x_{0}, \dots, x_{\ell}) = (x_{0}, \dots, \hat{x}_{i}, \dots, x_{\ell}),
\]
where $x_{j}$ is a singular point of $(x_{0}, \dots, x_{\ell})$ for $0 \leq j \leq i-1$, and $x_{i}$ is its smooth point.
\begin{lem}\label{finj}
If $\gir_{x}(G) \geq  5$, the above map $f_{\ell}$ is injective.
\end{lem}
\begin{proof}
Suppose that 
\[
f_{\ell}(x_{0}, \dots, x_{\ell}) = (x_{0}, \dots, \hat{x}_{i}, \dots, x_{\ell}) = (y_{0}, \dots, \hat{y}_{j}, \dots, y_{\ell}) = f_{\ell}(y_{0}, \dots, y_{\ell}).
\]
Then we have $d(x_{i-1}, x_{i+1}) = 2$ and $d(y_{j-1}, y_{j+1}) = 2$. Because the other pairs of adjacent points are apart from each other by distance $1$, we obtain $i = j$. If $i = j \geq 2$, then we have $x_{k}, y_{k} \in \{x_{0}, x_{1}\} = \{y_{0}, y_{1}\}$ for $0 \leq k \leq i = j$ by Lemma \ref{backandforth}. Then we obtain $x_{i} = y_{i}$ by the assumption that $x_{i -1} = y_{i-1}$, that is, $(x_{0}, \dots, x_{\ell}) =(y_{0}, \dots, y_{\ell})$. Suppose that  $i = j = 1$ and $x_{1} \neq y_{1}$. Then we have $x_{0} = y_{0}$, $x_{2} = y_{2}$, $d(x_{0}, x_{2}) = 2$, and $d(x_{0}, x_{1}) = d(x_{1}, x_{2}) = d(x_{0}, y_{1}) = d(y_{1}, x_{2}) = 1$ (see Figure~\ref{fig2}). Hence these four points form a $4$-cycle containing $x$, which is not the case. Thus we obtain $x_{1} = y_{1}$, that is, $(x_{0}, \dots, x_{\ell}) =(y_{0}, \dots, y_{\ell})$.
\begin{figure}[H]
\centering
\begin{tikzpicture}
\coordinate (a) at (0, 0) node [below] at (a) {$x_{0} = y_{0}$};
\coordinate (b) at (1, 1) node [right] at (b) {$x_{1}$};
\coordinate (c) at (-1, 1) node [left] at (c) {$y_{1}$};
\coordinate (d) at (0, 2) node [above] at (d) {$x_{2} = y_{2}$};
\fill (a) circle (3pt);
\fill (b) circle (3pt);
\fill (c) circle (3pt);
\fill (d) circle (3pt);
\draw[thick] (a) -- (b);
\draw[thick] (b) -- (d);
\draw[thick] (c) -- (d);
\draw[thick] (a) -- (c);
\end{tikzpicture}
\caption{An illustration of a $4$-cycle containing $x$.}
\label{fig2}
\end{figure}
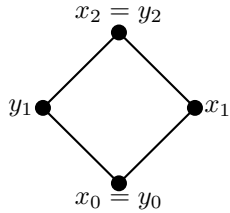

\end{proof}
By Lemma \ref{finj}, we can define a matching $M_{f_{\ell}}$ to $C_{\ast}$ by the injective map $f_{\ell}$. When $T_{\ell}$ is empty, we define the empty matching.
\begin{lem}\label{fmorse}
If $\gir_{x}(G) \geq  5$, then the above matching $M_{f_{\ell}}$ is a Morse matching.
\end{lem}
\begin{proof}
Let 
\[
(x_{0}, \dots, \hat{x}_{i}, \dots, x_{\ell}) \in \MC^{x}_{\ell-1, \ell}(G),
\]
where $x_{i}$ is a smooth point of the tuple 
\[
(x_{0},\dots, x_{\ell})\in \MC^{x}_{\ell, \ell}(G),
\]
but not the first one. Note that $i \geq 2$. We show that the tuple $(x_{0}, \dots, \hat{x}_{i}, \dots, x_{\ell})$ is not in the image of $f_\ell$, which implies that the directed graph $\Gamma^{M_{f_\ell}}_{C_{\ast}}$ is acyclic. Suppose that $f_{\ell}(y_{0}, \dots, y_{\ell}) = (x_{0}, \dots, \hat{x}_{i}, \dots, x_{\ell})$, and let $y_{j}$ be the first smooth point of the tuple $(y_{0}, \dots, y_{\ell})$. Then we have 
\[
(y_{0}, \dots, \hat{y}_{j}, \dots, y_{\ell}) = (x_{0}, \dots, \hat{x}_{i}, \dots, x_{\ell}),
\]
hence we have $i = j$ by the same argument in the proof of Lemma \ref{finj}. 
Then we also have $x_{i} \neq y_{j}$. Because $y_{j}$ is the first smooth point, we have $\{y_{0}, \dots, y_{j}\} = \{y_{0}, y_{1}\}$ by Lemma \ref{backandforth}. By the assumption that $x_{k} = y_{k}$ for $0 \leq k \leq i-1$, we obtain that $y_{i} = y_{i-2} = x_{i-2}$. Because $y_{i}$ is adjacent to $y_{i+1} = x_{i+1}$, we have $d(x_{i-2}, x_{i+1}) = d(x_{i+1}, x_{i}) = d(x_{i}, x_{i-1}) = d(x_{i-1}, x_{i-2}) = 1$ (see Figure~\ref{fig3}). Then there is a $3$- or  $4$-cycle containing an edge $\{y_{i-2}, y_{i-1}\} = \{y_{0}, y_{1}\} = \{x_{0}, x_{1}\}$ unless we have $x_{i-1} = x_{i+1}$. The former case contradicts that $\gir_x(G)\ge5$. The latter case contradicts the fact that $x_{i}$ is a smooth point of $(x_0,\ldots,x_\ell)\in\MC^x_{\ell,\ell}(G)$.
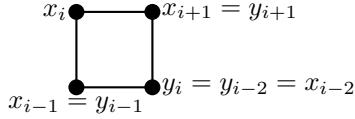
\begin{figure}[H]
\centering
\begin{tikzpicture}
\coordinate (a) at (0, 0) node [below] at (a) {$x_{i-1} = y_{i-1}$};
\coordinate (b) at (0, 1) node [left] at (b) {$x_{i}$};
\coordinate (c) at (1, 0) node [right] at (c) {$y_{i} = y_{i-2} = x_{i-2}$};
\coordinate (d) at (1, 1) node [right] at (d) {$x_{i+1} = y_{i+1}$};
\fill (a) circle (3pt);
\fill (b) circle (3pt);
\fill (c) circle (3pt);
\fill (d) circle (3pt);
\draw[thick] (a) -- (b);
\draw[thick] (b) -- (d);
\draw[thick] (c) -- (d);
\draw[thick] (a) -- (c);
\end{tikzpicture}
\caption{An illustration of a $3$- or $4$-cycle containing $x$ in the case that $x_{i-1}\neq x_{i+1}$. A $3$-cycle appears when $x_i$ and $y_i$ are adjacent, otherwise a $4$-cycle appears.}
\label{fig3}
\end{figure}

\end{proof}
\begin{proof}[Proof of Theorem \ref{lpart}]
By Lemma \ref{fmorse}, the chain complex $C_{\ast}$ is homotopy equivalent to the chain complex generated by the unmatched generators of the Morse matching $M_{f_\ell}$.
The unmatched generators in $\MC^{x}_{\ell, \ell}(G)$ are exactly the tuples that have only singular points, and by Lemma \ref{backandforth}, they are of the form $(x, y, x, y, \dots)$, where $y$ is adjacent to $x$. Because the differential of $\MC^{x}_{\ast, \ell}(G)$ vanishes on these generators, $\MH^{x}_{\ell, \ell}(G)$ is isomorphic to a free module generated by the tuples of the form $(x,y,x,y,\ldots)$. This completes the proof.
\end{proof}
\subsection{Computation for 
non-diagonal part}\label{alg3}
We extend our matching constructed above to a larger part of magnitude chain complex. For a tuple $(x_{0}, \dots, x_{n}) \in V(G)^{n+1}$, we call $(x_{g}, x_{g+1})$ {\it a gap} if $d(x_{g}, x_{g+1}) \geq 2$, and we call it {\it the first gap} if additionally $d(x_{j}, x_{j+1}) = 1$ for $0\leq j \leq g-1$. For $0\leq i \leq \ell-1$, let $T_{\ell - i}$ be a subset of $\MC^{x}_{\ell-i, \ell}(G)$ defined as
\[
T_{\ell-i}\coloneqq
\left\{(x_0,\ldots,x_{\ell-i})\in \MC^x_{\ell-i,\ell}(G)\left|
\begin{array}{l}
\text{some $x_j$'s are smooth point for $1\le j\le g-1$,}\\
\text{where $(x_g,x_{g+1})$ is the first gap}
\end{array}
\right.\right\}
\]
for $i \geq 1$, and the subset $T_{\ell}$ defined in the previous subsection for $i = 0$. We simply say that $x_{j}$ {\it is the first smooth point before the first gap of} $(x_{0}, \dots, x_{\ell-i})$ if $x_{j}$ with $1\leq j \leq g-1$ is a smooth point and $x_k$'s are singular points for $0 \leq k \leq  j-1$, where $(x_g, x_{g+1})$ is the first gap. For $i = 0$, we mean just the first smooth point. Whenever $T_{\ell-i}\neq\emptyset$, we define a map 
\[
f_{\ell - i}\colon T_{\ell - i} \longrightarrow \MC^{x}_{\ell-i-1, \ell}(G)
\]
by deleting the first smooth point before the first gap, that is,
\[
f_{\ell-i}(x_{0}, \dots, x_{\ell - i}) = (x_{0}, \dots, \hat{x}_{j}, \dots, x_{\ell - i}),
\]
where $x_{j}$ is the first smooth point of $(x_{0}, \dots, x_{\ell-i})$ before the first gap. Note that our definition of $T_{\ell - i}$'s and $f_{\ell - i}$'s contain those of $f_{\ell}$ and $T_{\ell}$ defined in the previous subsection, respectively, by considering $i=0$.
The image of the map $f_{\ell -i}$ is disjoint from the subset $T_{\ell - i -1}$ for $0 \leq i \leq \ell - 1$ since the deletion of a point by  $f_{\ell -i}$ makes a new first gap before which there exists no smooth points.
\begin{lem}\label{fiinj}
If $\gir_{x}(G) \geq  5$, then $f_{\ell - i}$ is injective for $0 \leq i \leq \ell -1$.
\end{lem}
\begin{proof}
As shown in Lemma~\ref{finj}, $f_\ell$ is injective. Hence, we assume that $i\ge1$.
Suppose that $f_{\ell-i}(x_{0}, \dots, x_{\ell - i}) = f_{\ell - i}(y_{0}, \dots, y_{\ell - i})$. By the same argument in the proof of Lemma \ref{finj}, the positions of the first smooth point and the first gap of the both tuples are same. By looking at the parts before the first gap, the statement follows from the same argument in the proof of Lemma \ref{finj}.
\end{proof}
By Lemma \ref{fiinj}, we can define a matching $M_{f_{\ast}}$ of $\MC^{x}_{\ast, \ell}(G)$ by the injective maps $f_*=(f_{\ell - i})_{0 \leq i \leq \ell - 1}$. In the following, we assume $i$ to be in the range $0 \leq i \leq \ell - 1$ unless otherwise mentioned.
\begin{lem}\label{fimorse}
If $\gir_{x}(G) \geq  5$, then the above matching $M_{f_{\ast}}$ is a Morse matching.
\end{lem}
\begin{proof}
Let 
\[
(x_{0}, \dots, \hat{x}_{j}, \dots, x_{\ell - i}) \in \MC^{x}_{\ell-i -1, \ell}(G),
\]
where $x_{j}$ is a smooth point of the tuple 
\[
(x_{0},\dots, x_{\ell - i})\in \MC^{x}_{\ell - i, \ell}(G),
\]
but not the first smooth point before the first gap. The case for $i = 0$ has been already considered in Lemma \ref{fmorse}, hence we assume $i \geq 1$. Let $(x_{g}, x_{g+1})$ be the first gap of the tuple $(x_{0},\dots, x_{\ell - i})$. If $j = g$ or $g + 1$, then $(x_{0}, \dots, \hat x_{j}, \dots, x_{\ell - i})$ is not in the image of $f_{\ell - i}$. It is because the first gap $(x_{g-1}, x_{g+1})$ or $(x_{g}, x_{g+2})$ of $(x_{0}, \dots, \hat x_{j}, \dots, x_{\ell - i})$ must satisfy that $d(x_{g-1}, x_{g+1}) \geq 3$ or $d(x_{g}, x_{g+2}) \geq 3$ respectively, while the first gap of an image of $f_{\ell -i}$ must have distance 2.
For the case that $j \leq g-1$, we can show that the tuple $(x_{0}, \dots, \hat x_{j}, \dots, x_{\ell - i})$ is not in the image of $f_{\ell - i}$ by the same argument in the proof of Lemma \ref{fmorse}. Hence the remained case is that $j \geq g+2$. In this case, if we have 
\[
(x_{0}, \dots, \hat{x}_{j}, \dots, x_{\ell - i}) = f_{\ell - i}(y_{0}, \dots, y_{\ell - i}),
\]
the tuple $(y_{0}, \dots, y_{\ell - i})$ must be of the form
\[
(x_{0}, \dots, x_{g}, y_{\rm new}, x_{g+1}, \dots, x_{j-1}, x_{j+1}, \dots, x_{\ell - i})
\]
with $d(x_{g}, x_{g+1}) = 2$  and $y_{\rm new}$ is the first smooth point before the first gap by the definition of $f_{\ell -i}$. Then the first gap $(y_{g'}, y_{g' + 1})$ of $(y_{0}, \dots, y_{\ell - i})$ satisfies $g' \geq g+1$. Hence there cannot be a cycle of the form 
\[
a_{1} \longrightarrow b_{1} \longrightarrow \cdots \longrightarrow a_{p}  \longrightarrow b_{p}  \longrightarrow a_{1},
\]
in $\Gamma^{M_{f_{\ast}}}_{\MC^{x}_{\ast, \ell}(G)}$ with $a_{k} \in \MC^{x}_{\ell - i, \ell}(G)$, $b_{k} \in \MC^{x}_{\ell - i - 1, \ell}(G)$ because the position of the first gap of $a_{k}$ moves backward.
This completes the proof.
\end{proof}
By Lemma \ref{fimorse}, we obtain a chain complex $(\mathring{\MC}^{x}_{\ast, \ell}(G), \mathring\partial_{\ast,\ell})$
consisting of unmatched generators by the Morse matching $M_{f_*}$, which is homotopy equivalent to the original magnitude chain complex $(\MC^{x}_{\ast, \ell}(G), \partial_{\ast,\ell})$.
The following lemma characterizes the generators of $(\mathring{\MC}^{x}_{\ast, \ell}(G), \mathring\partial_{\ast,\ell})$.
\begin{lem}\label{unmatched}
Let $\gir_{x}(G) \geq  5$. A tuple $(x_{0}, \dots, x_{\ell - i}) \in \MC^{x}_{\ell - i, \ell}(G)$ is unmatched by the matching $M_{f_{\ast}}$ if and only if  it satisfies one of the following conditions$:$
\begin{enumerate}
\item[(i)] It has no gaps and no smooth points,
\item[(ii)] It has the first gap $(x_{g}, x_{g+1})$ with $g \geq 1$ and $d(x_{g}, x_{g+1}) \geq 3$ such that there is no smooth point before the first gap, 
\item[(iii)] It has the first gap $(x_{g}, x_{g+1})$ with $g \geq 1$ and $d(x_{g}, x_{g+1}) = 2$ such that there is no smooth point before the first gap. Furthermore, every vertex $z$ adjacent to both of $x_{g}$ and $x_{g+1}$ is the second smooth point of $(x_{0}, \dots, x_{g}, z, x_{g+1}, \dots, x_{\ell - i})$,
\item[(iv)] It has the first gap $(x_{0}, x_{1})$ with $d(x_{0}, x_{1}) \geq 3$.
\end{enumerate}
\end{lem}
\begin{proof}
Let $(x_{0}, \dots, x_{\ell - i}) \in \MC^{x}_{\ell - i, \ell}(G)$ satisfy none of the above conditions. We will show that $(x_{0}, \dots, x_{\ell - i})$ is matched. If there is a smooth point before the first gap, then it is in $T_{\ell -i}$, hence it is matched. Hence we can suppose that $(x_{0}, \dots, x_{\ell - i})$ has the first gap $(x_{g}, x_{g+1})$ with $g \geq 0$ and $d(x_{g}, x_{g+1}) = 2$ such that there is no smooth point before the first gap, and furthermore, there is a vertex $z$ adjacent to $x_{g}$ and $x_{g+1}$ such that $z$ is the first smooth point before the first gap of $(x_{0}, \dots, x_{g}, z, x_{g+1}, \dots, x_{\ell - i})$. Then we have
\[
f_{\ell - i+1}(x_{0}, \dots, x_{g}, z, x_{g+1}, \dots, x_{\ell - i}) = (x_{0}, \dots, x_{\ell - i}) ,
\]
hence it is matched. Therefore the above conditions are necessary to be unmatched. The sufficiency is straightforward.
\end{proof}

Now we look at the differential $\mathring{\partial}_{\ast,\ell}$ on $\mathring{\MC}^{x}_{\ast, \ell}(G)$.
\begin{lem}\label{dpath}
Let $\gir_{x}(G) \geq  5$. Let $\alpha$ be a tuple satisfying one of the conditions in Lemma \ref{unmatched}. Then there are no paths of length $\geq 2$ in $\Gamma^{M_{f_{\ast}}}_{\MC^{x}_{\ast, \ell}(G)}$ that start from $\alpha$.
\end{lem}
\begin{proof}
Note that there exist no directed edges $\a\longrightarrow\beta$ such that $\alpha \in \MC^{x}_{\ell - i, \ell}(G)$, $\beta \in \MC^{x}_{\ell - i + 1, \ell}(G)$ by Lemma~\ref{unmatched}.
Hence, let $\alpha \longrightarrow \beta$ be a directed edge in $\Gamma^{M_{f_{\ast}}}_{\MC^{x}_{\ast, \ell}(G)}$ with $\alpha \in \MC^{x}_{\ell - i, \ell}(G)$, $\beta \in \MC^{x}_{\ell - i - 1, \ell}(G)$.
In order that this directed edge is extended to a path of length $2$, $\beta$ must be in the image of $f_{\ell - i}$. Note that, in order to be in the image of $f_{\ell - i}$, $\beta$ must have the first gap with distance exactly $2$. Hence $\alpha$ and  $\beta$ must be of the form
\begin{align*}
\alpha &= (x_{0}, \dots, x_{g}, x_{g+1}, \dots,x_k,\ldots, x_{\ell-i}), \\
\beta &= (x_{0}, \dots, x_{g}, x_{g+1}, \dots, \hat{x}_{k}, \dots, x_{\ell - i}),
\end{align*}
where $(x_g,x_{g+1})$ is the first gap of $\a$ and $\beta$ with $g\geq 0$, $d(x_{g}, x_{g+1}) = 2$, and $g+2 \leq k \leq \ell - i - 1$. Further, $\alpha$ must satisfy (iii) of Lemma \ref{unmatched} by the assumption. Hence every vertex $y$ adjacent to both of $x_{g}$ and $x_{g+1}$ is the second smooth point of the tuple 
\[
(x_{0}, \dots, x_{g}, y, x_{g+1}, \dots, \hat x_{k}, \dots, x_{\ell - i}),
\]
which implies that $\beta$ cannot be in the image of $f_{\ell - i}$. Hence the statement follows.
\end{proof}
We obtain the following by Lemma \ref{dpath} and Theorem \ref{morsefact}.
\begin{lem}\label{isom}
Let $\gir_{x}(G) \geq  5$. The differentials on $\mathring{\MC}^{x}_{\ast, \ell}(G)$ are restrictions of those on~$\MC^{x}_{\ast, \ell}(G)$.
\end{lem}
Now we further construct a Morse matching for $(\mathring{\MC}^{x}_{\ast, \ell}(G), \mathring\partial_{\ast,\ell})$. Before that, we study some properties of the unmatched tuples of the matching $M_{f_{\ast}}$ by the following three lemmas.
\begin{lem}\label{smoothsing}
Suppose that $\gir_{x}(G) \geq 5$. Let
\[
(x_{0}, \dots, x_{g}, x_{g+1}, \dots, x_{\ell - i}) \in \mathring{\MC}^{x}_{\ell - i, \ell}(G),
\]
which satisfies the condition (ii) or (iii) in Lemma \ref{unmatched} with the first gap $(x_{g}, x_{g+1})$, $g\geq 1$. If $x_{g}$ is its smooth point, then $x_{g-1}$ is a singular point of the tuple $(x_{0}, \dots, x_{g-1}, \hat{x}_{g}, x_{g+1}, \dots, x_{\ell - i})$.
\end{lem}
\begin{proof}
By Lemma \ref{backandforth}, we have $x_{2m} = x_{2m+2}$ and $x_{2m+1} = x_{2m+3}$ for $0 \leq 2m \leq 2m+3 \leq g$. Since $x_{g}$ is a smooth point, we have that $d(x_{g-1}, x_{g+1}) = d(x_{g-1}, x_{g}) + d(x_{g}, x_{g+1})$. Then we have that 
\[
d(x_{g-2} = x_{g}, x_{g-1}) + d(x_{g-1}, x_{g+1}) = d(x_{g}, x_{g+1}) + 2d(x_{g-1}, x_{g}) > d(x_{g-2}=x_{g}, x_{g+1}).
\]
Hence $x_{g-1}$ is a singular point of the tuple $(x_{0}, \dots, x_{g-2}, x_{g-1}, \hat{x}_{g}, x_{g+1}, \dots, x_{\ell - i})$.
\end{proof}

\begin{lem}\label{distancetwocase}
Let 
\[
(x_{0}, \dots, x_{g}, x_{g+1}, \dots, x_{\ell - i}) \in \mathring{\MC}^{x}_{\ell - i, \ell}(G),
\]
which satisfies the condition (iii) in Lemma \ref{unmatched} with the first gap $(x_{g}, x_{g+1})$. If $\gir_{x}(G) > 5$, then $x_{g}$ is a smooth point of $(x_0,\ldots,x_g,x_{g+1},\ldots,x_{\ell-i})$.
\end{lem}
\begin{proof}
Note that $x \in \{x_{g-1}, x_{g}\} $ by the same argument as that in Lemma \ref{backandforth}. Assume that $x_{g}$ is a singular point of $(x_0,\ldots,x_g,x_{g+1},\ldots,x_{\ell-i})$. Let $z$ be a vertex adjacent to $x_{g}$ and $x_{g+1}$. Then we have $d(x_{g-1}, z) = d(x_{g-1}, x_{g}) + d(x_{g}, z) = 2$ so that it satisfies (iii) of Lemma \ref{unmatched}. Hence we have $x_{g-1} \neq z$. Since $x_{g}$ is a singular point of $(x_0,\ldots,x_g,x_{g+1},\ldots,x_{\ell-i})$, we have
\[
d(x_{g-1}, x_{g+1}) < d(x_{g-1}, x_{g}) + d(x_{g}, x_{g+1}) = 3.
\]
If $d(x_{g-1}, x_{g+1}) = 2$, then there exists a $5$-cycle containing $x$ because the point adjacent to $x_{g-1}$ and $x_{g+1}$ do not coincide with $x_{g}$ or $z$. This contradicts the assumption.
If $d(x_{g-1}, x_{g+1}) = 1$, then there exists a $4$-cycle containing $x$. Further, we have $d(x_{g-1},x_{g+1})\neq0$ because $d(x_g,x_{g-1})=1$ and $d(x_g,x_{g+1})=2$. Therefore, we conclude that $x_{g}$ can never be a singular point.
\end{proof}

\begin{lem}\label{singsmooth}
Let $i\ge1$. Let 
\[
(x_{0}, \dots, x_{g}, x_{g+1}, \dots, x_{\ell - i}) \in \mathring{\MC}^{x}_{\ell - i, \ell}(G),
\]
which satisfies the condition (ii) or (iv) in Lemma~\ref{unmatched} with the first gap $(x_{g}, x_{g+1})$, $g\geq 0$. Suppose that  $x_{g}$ is a singular point of $(x_0,\ldots,x_g,x_{g+1},\ldots,x_{\ell-i})$. If $\gir_{x}(G) \geq 2i + 4$, then
\[
(x_{0}, \dots, x_{g}, y, x_{g+1}, \dots, x_{\ell - i}) \in \mathring{\MC}^{x}_{\ell - i + 1, \ell}(G),
\]
where $y$ is taken as $x_{g-1}$ for $g \geq 1$ and as an arbitrary vertex adjacent to $x_0$ that lies in a shortest path connecting $x_0$ and $x_1$ for $g = 0$.
\end{lem}
\begin{proof}
Let 
\[
x_{g} = p_{0} \longrightarrow \cdots \longrightarrow p_{d(x_{g}, x_{g+1})} = x_{g+1}
\]
be a shortest path connecting $x_{g}$ and $x_{g+1}$. When $g = 0$, we can take $y = p_{1}$ so that $y$ becomes a smooth point. If we have $g \geq 1$ and $p_{1} = x_{g-1}$, then we can take $y = p_{1} = x_{g-1}$ so that $d(x_{g}, x_{g+1}) = d(x_{g}, x_{g-1}) + d(x_{g-1}, x_{g+1})$. Hence we suppose that $g \geq 1$ and $p_{1} \neq x_{g-1}$. Since $x_{g}$ is a singular point of $(x_0,\ldots,x_g,x_{g+1},\ldots,x_{\ell-i})$, there exist a shortest path 
\[
x_{g-1} = q_{0} \longrightarrow \cdots \longrightarrow q_{N} = x_{g+1}
\]
with $N < 1 + d(x_{g}, x_{g+1})$ and $q_{1} \neq x_{g}$. Let $j$ be the minimum number such that $q_{j}$ coincides with some $p_{m}$. Then 
\[
x_{g} \longrightarrow x_{g-1} = q_{0} \longrightarrow \cdots \longrightarrow q_{j} = p_{m} \longrightarrow p_{m-1} \longrightarrow p_{0} = x_{g}
\]
is a cycle of length $< 2d(x_{g}, x_{g+1}) + 2 $ because $(j,m)\neq(1,0),(0,1)$.
Note that we have $d(x_{g}, x_{g+1}) \leq i+1$ because $L(x_{0}, \dots, x_{g}, x_{g+1}, \dots, x_{\ell - i}) = \ell$. Hence the obtained cycle has length $< 2i+4$. Since $x_{0}, \dots, x_{g-1}$ are all singular points, we have $x_{g-1} = x_{0}$ or $x_{g} = x_{0}$ by Lemma \ref{backandforth}. Therefore this cycle contains $x$ as its vertex, it contradicts that $\gir_{x}(G) \geq 2i + 4$. Finally, we show that obtained tuple $(x_{0}, \dots, x_{g}, y, x_{g+1}, \dots, x_{\ell - i}) $ is unmatched by the matching $M_{f_{\ast}}$. 
\begin{itemize}
\item If $d(x_{g}, x_{g+1}) \geq 4$, then we have $d(y, x_{g+1}) \geq 3$, hence it satisfies (ii) of Lemma \ref{unmatched}. 
\item If $d(x_{g}, x_{g+1}) = 3$ and $g = 0$, then we have $d(y, x_{1}) = 2$. Let $z$ be a vertex adjacent to both of $y$ and $x_{1}$. Then we must have $d(x_{0}, z) = d(x_{0}, y) + d(y, z)$ because $x = x_{0}$ and there is no $3$-cycle containing $x$. Hence the tuple $(x_0,y,x_1,\ldots,x_{\ell-i})$ satisfies (iii) of Lemma \ref{unmatched}.
\item If $d(x_{g}, x_{g+1}) = 3$ and $g \geq 1$, then we have $d(y, x_{g+1}) = 2$ with $y = x_{g-1}$. Let $z$ be a vertex adjacent to both of $y$ and $x_{g+1}$. Then we must have $d(x_{g}, z) = d(x_{g}, y) + d(y, z)$ because either of $x_{g}$ or $y = x_{g-1}$ coincides with $x$, and there are no $3$-cycles containing $x$. Hence the tuple $(x_{0}, \dots, x_{g}, y, x_{g+1}, \dots, x_{\ell - i}) $ satisfies (iii) of Lemma \ref{unmatched}.\qedhere
\end{itemize}
\end{proof}

Now we consider the following truncated chain complex  for $i \geq 1$:
\[
0 \longrightarrow \mathring{\MC}^{x}_{\ell, \ell}(G) \longrightarrow \mathring{\MC}^{x}_{\ell -1, \ell}(G) \longrightarrow \cdots \longrightarrow \mathring{\MC}^{x}_{\ell - i - 1, \ell}(G) \longrightarrow 0.
\]
We denote this chain complex by $D_{\ast}$ in the following. Let $U_{\ell - j}$ be the subset of generators of $\mathring{\MC}^{x}_{\ell - j, \ell}(G)$ which consists of all the tuples satisfying (ii) or (iii) in Lemma \ref{unmatched} with smooth point $x_{g}$. We define maps 
\[
h_{\ell - j}\colon U_{\ell - j} \longrightarrow \mathring{\MC}^{x}_{\ell - j - 1, \ell}(G)
\]
for $1 \leq j \leq i$ by
\[
h_{\ell - j}(x_{0}, \dots, x_{g}, x_{g+1}, \dots, x_{\ell - j}) = (x_{0}, \dots, \hat{x}_{g}, x_{g+1}, \dots, x_{\ell - j}),
\]
where $(x_{g}, x_{g+1})$ is the first gap. By Lemma \ref{smoothsing}, the image of $h_{\ell - j}$ is disjoint from $U_{\ell - j - 1}$.
\begin{lem}\label{hjinj}
Let $i\ge1$. If $\gir_{x}(G) \geq 2i + 5$, then $h_{\ell -j}$ is injective for $1 \leq j \leq i$.
\end{lem}
\begin{proof}
Suppose that $h_{\ell - j}(x_{0}, \dots, x_{\ell-j}) = h_{\ell - j}(y_{0}, \dots, y_{\ell-j})$. We can verify that the position of the first gaps of $(x_{0}, \dots, x_{\ell-j})$ and $(y_{0}, \dots, y_{\ell-j})$ are identical in the same manner as in Lemma~\ref{finj}. 
Then we have $x_{k} = y_{k}$ except for $k = g$, where $(x_g,x_{g+1})$ and $(y_g,y_{g+1})$ are the first gaps. Since $x_{k}$ and $y_{k}$ are singular points of $(x_{0}, \dots, x_{\ell-j})$ and $(y_{0}, \dots, y_{\ell-j})$, respectively, for $0 \leq k \leq g-1$, we have $\{x_{0}, \dots, x_{g}\} = \{x_{0}, x_{1}\}$ and $\{y_{0}, \dots, y_{g}\} = \{y_{0}, y_{1}\}$ by Lemma \ref{backandforth}. Hence we obtain $x_{g} = y_{g}$ if $g \geq 2$. Suppose that $g = 1$ and $x_{1} \neq y_{1}$. Since $x_1$ and $y_{1}$ are smooth points of $(x_{0}, \dots, x_{\ell-j})$ and $(y_{0}, \dots, y_{\ell-j})$, respectively, there exist shortest paths
\begin{align*}
x=x_{0} \longrightarrow x_{1} \longrightarrow \cdots \longrightarrow x_{2} = y_{2}
\shortintertext{and}
x=y_{0} \longrightarrow y_{1} \longrightarrow \cdots \longrightarrow x_{2} = y_{2}
\end{align*}
of length $ 1 + d(x_{1}, x_{2})=1 + d(y_{1}, y_{2}) \leq j + 2$. Then there exists a cycle of length $\leq 2(j+2) \leq 2i + 4$ containing $x$ as its vertex, which contradicts the assumption. Hence we obtain that $(x_{0}, \dots, x_{\ell-j}) = (y_{0}, \dots, y_{\ell-j})$.
\end{proof}

By Lemmas~\ref{smoothsing} and~\ref{hjinj}, we can define a matching $M_{h_{\ast}}$ of $D_{\ast}$ by injective maps $h_*=(h_{\ell - j})_{1\le j\le i}$. 
\begin{lem}\label{hjmorse}
Let $i\ge1$. If $\gir_{x}(G) \geq  2i + 5$, then the above matching $M_{h_{\ast}}$ is a Morse matching.
\end{lem}
\begin{proof}
By Lemma \ref{isom}, any differentials corresponding to edges in $M_{h_{\ast}}$ are isomorphisms (cf. Definition \ref{morsematching}). Let 
\[
(x_{0}, \dots, x_{g}, x_{g+1}, \dots, x_{\ell-j}) \in \mathring{\MC}^{x}_{\ell - j, \ell}(G)
\]
with the first gap $(x_{g}, x_{g+1})$, $g\geq 0$. Let 
\[
(x_{0}, \dots, x_{g}, x_{g+1}, \dots, x_{\ell-j}) = a_{1}  \longrightarrow b_{1} \longrightarrow a_{2} \longrightarrow b_{2} \longrightarrow \cdots
\]
be a path in $\Gamma^{M_{h_{\ast}}}_{D_{\ast}}$ with $a_{p} \in \mathring{\MC}^{x}_{\ell - j, \ell}(G)$ and $b_{p} \in \mathring{\MC}^{x}_{\ell - j - 1, \ell}(G)$ for $p \in \N$. Here the directed edge $a_{p} \longrightarrow b_{p}$ corresponds to a directed edge in $\Gamma^{M_{f_{\ast}}}_{C_{\ast}}$. Again by Lemma \ref{isom}, $b_1$ is obtained by deleting some smooth point of $a_1$. Hence $b_{1}$ must be of the form
\[
(x_{0}, \dots, x_{g}, x_{g+1}, \dots, \hat{x}_{k}, \dots, x_{\ell - j})
\]
with $g+1 \leq k \leq \ell - j - 1$, and $x_{g}$ must be its singular point to be in the image of $h_{\ell - j}$ by Lemma~\ref{smoothsing}.
It follows that $a_{2}$ is of the form 
\[
(x_{0}, \dots, x_{g}, y, x_{g+1}, \dots, \hat{x}_{k}, \dots, x_{\ell - j}),
\]
where $(y,x_{g+1})$ is the first gap.
Inductively, we conclude that the first gap of $a_i$ moves backward as $i$ increases. Hence there cannot be any cycle in $\Gamma^{M_{h_{\ast}}}_{D_{\ast}}$.
\end{proof}

\begin{proof}[Proof of Theorem~\ref{otherpart}]
By Lemma \ref{hjmorse}, the chain complex $D_{\ast}$ is homotopy equivalent to the chain complex consisting of all the unmatched tuples by $M_{h_{\ast}}$. By Lemma~\ref{distancetwocase}, any tuples satisfying the condition (iii) in Lemma \ref{unmatched} are matched. By Lennma~\ref{singsmooth}, any tuples satisfying the condition (ii) or (iv) in Lemma \ref{unmatched} are matched. Hence it turns out that the unmatched tuples by $M_{h_{\ast}}$ are only those satisfying the condition (i) of Lemma \ref{unmatched} except for the tuples in $\MC^{x}_{\ell - i - 1, \ell}(G)$. Hence the statement follows.
\end{proof}

\subsection{A criterion for diagonality}\label{alg4}
We devote this subsection to proving Theorem~\ref{nondiag} which gives a criterion of the diagonality of graphs. First we recall the definition of the local girth of a graph at a fixed edge, as seen in the introduction.
\begin{df}
Let $G$ be a graph and $e \in E(G)$ be an edge. We define the {\it local girth of} $G$ {\it at} $e$ by 
\begin{align*}
\gir_e(G)\coloneqq\inf\{i\ge3\mid\text{ there exists an $i$-cycle  in $G$ containing $e$ as its edge}\}.
\end{align*}
\end{df}

\begin{proof}[Proof of Theorem \ref{nondiag}]
We first prove for the case that $k$ is odd. We put $k = 2K+1$. Let $1, 2, \dots, 2K+1$ be vertices of a $(2K+1)$-cycle with $e = \{1, 2\}$. We suppose that each vertex $i$ is adjacent to vertices $i-1$ and $i+1$, where we put $0 = 2K+1$ and $2K+2 = 1$. Note that the distance between each pair of vertices of this cycle in $G$ is identical to that of the cycle graph itself. If not, there will be cycles of length $< 2K+1$ containing $e$, which contradicts the assumption. In particular, we have $d(1, K+2) = d(2, K+2) = K$. We show that the homology cycle
\[
[(1, 2, K+2)] \in \MH^{1, K+2}_{2, K+1}(G)
\]
is non-trivial. 

Assume that we have $[(1, 2, K+2)] = 0$, that is, there exist not necessarily distinct tuples $\alpha_{1}, \dots, \alpha_{n} \in \MC^{1, K+2}_{3, K+1}(G)$ and a vertex $a \in V(G)$ such that 
\[
\partial  \Big( (1, 2,a, K+2) + (-1)^{s_1} \alpha_{1} +  \dots  + (-1)^{s_{n}}\alpha_{n} \Big) = (1, 2, K+2).
\]
Here, $s_1,\ldots,s_n\in\{0,1\}$ and we set $s_0 = 0$. Note that any tuples of the form $(1, a, 2, K+2)$ do not appear in $\alpha_i$'s, because $L(1, a, 2, K+2) > K+1$. We put $\alpha_{0} = (1, 2, a, K+2)$ and $\alpha_{i} = (1, x_{i}, y_{i}, K+2)$ for $i\in\{1,\dots,n\}$. 

Now we construct a graph $A(G)$ with vertices $\{2, a, x_{1}, y_{1},  \dots, x_{n}, y_{n}\}$. We span an edge between $v, w$ if $(1, v, w, K+2) = \alpha_{i}$ or $(1, w, v, K+2) = \alpha_{i}$ for some $i$. Then we have the following lemma. In the following, we denote by $\langle v_1, \dots, v_n \rangle$ a path in a graph consisting of edges $\{v_1, v_2\}, \dots ,\{v_{n-1}, v_n\}$ in this order to make it easy to distinguish between paths and tuples. 
\begin{lem}\label{2nd}
Let $x$ be a vertex of $A(G)$ which is connected to the vertex $2$. Let $\langle1, b_{1}, \dots, x\rangle$ be a shortest path in $G$ connecting 1 and $x$. Then $b_{1} = 2$. 
\end{lem}
\begin{proof}
Let $\langle 2, a_{1}, a_{2}, \dots, x = a_{N} \rangle$ be a path in $A(G)$ connecting 2 and $x$. Note that $a_{1}$ satisfies that $d(1, 2) + d(2, a_{1}) + d(a_{1}, K+2) = K + 1$ because $(1, 2, a_{1}, K + 2) = \alpha_{m}$ for some $m$.
Let $\langle 1, b^{i}_{1}, \dots, a_{i}\rangle$ be a shortest path in $G$ connecting 1 and $a_i$. We show that $b^{i}_{1} = 2$ by induction on $i$. If $b^{1}_1 \neq 2$, then a closed path obtained by concatenating three paths, $\langle 1, b^{1}_1, \dots, a_1 \rangle$, a shortest path connecting $a_1$ and 2, and the edge between 2 and 1 produces a cycle containing $e$. Note here that the shortest path from $2$ to $a_{1}$ does not pass through $1$. If it goes through $1$, then we have $K+1 = d(1, 2) + d(2, a_{1}) + d(a_{1}, K+2) = 2 + d(1, a_{1}) + d(a_{1}, K +2) \geq 2 + d(1, K+2) = K+2$. Because $d(1, 2) + d(2, a_1) \leq K$, the obtained cycle is of length $\leq 2K$, which contradicts the assumption. Hence we have $b^{1}_1 = 2$.

Suppose $b^{i}_{1} = 2$ and $b^{i+1}_{1} \neq 2$. If $(1, a_{i},  a_{i+1}, K+2) = \alpha_{m}$ for some $m$, then a closed path obtained by concatenating three paths, $\langle 1, b^{i}_1, \dots, a_i\rangle$, a shortest path connecting $a_i$ and $a_{i+1}$, and $\langle a_{i+1}, \dots, b^{i+1}_1, 1\rangle $ produces a cycle containing $e$.
Note here that the shortest path from $a_i$ to $a_{i+1}$ does not pass through $1$ in the same manner as discussed above.
Because $d(1, a_{i}) + d(a_{i}, a_{i+1}) \leq K$, the obtained cycle is of length $\leq 2K$, which contradicts the assumption. Similarly, if $(1, a_{i+1},  a_{i}, K+2) = \alpha_{m}$ for some $m$, then a closed path obtained by concatenating three paths, $\langle1, b^{i}_1, \dots, a_i\rangle$, a shortest path connecting $a_i$ and $a_{i+1}$, and $\langle a_{i+1}, \dots, b^{i+1}_1, 1\rangle $ produces a cycle containing $e$. Because $d(1, a_{i+1}) + d(a_{i+1}, a_{i}) \leq K$, the obtained cycle is of length $\leq 2K$, which also contradicts the assumption. Hence we have $b^{i+1}_1 = 2$.
\end{proof}
Now we divide the collection of tuples $ \alpha_{0} = (1, 2, a, K+2), \alpha_{1}, \dots, \alpha_{n}$ into subcollections 
\[
C_{0}, \dots, C_{M}
\]
corresponding to the connected components of $A(G)$. Namely, two tuples $\alpha_{i}$ and $\alpha_{j}$ belong to the same subcollection if the corresponding edges in $A(G)$ are connected by some path. We suppose that 
\[
(1, 2, a, K+2) \in C_{0}.
\]
Then we have 
\[
\partial \Big( \sum_{i\geq 1}\sum_{\alpha_{j} \in C_{i}} (-1)^{s_{j}}\alpha_{j} \Big) = 0.
\]
If not, there exists a tuple $(1, x, K+2) \neq (1, 2, K+2)$ which appears in the left-hand side, and also in $\partial \Big( \sum_{\alpha_{j} \in C_{0}} (-1)^{s_{j}}\alpha_{j} \Big)$ with the opposite sign, because the total sum is $ (1, 2, K+2)$.
Then it implies that the vertex $x$ in $A(G)$ belongs to two distinct connected components of $A(G)$, which is a contradiction. Hence we have
\[
\partial \Big( \sum_{\alpha_{j} \in C_{0}} (-1)^{s_{j}}\alpha_{j} \Big) = (1, 2, K+2),
\]
which implies that there exists a tuple $\alpha_{m} = (1, x_{m}, y_{m}, K+2) \in C_{0}$ such that $L(1, x_{m}, K+2) = K$ or $L(1, y_{m}, K+2) = K$, because the right-hand side consists of odd terms. If $L(1, x_{m}, K+2) = K$, then a path in $G$ obtained by concatenating a shortest path connecting $1$ and $x_{m}$, and a shortest path connecting $x_{m}$ and $K+2$ is a shortest path connecting $1$ and $K+2$. Because a shortest path connecting $1$ and $x_{m}$ goes through $2$ by Lemma \ref{2nd},  we have $d(2, K+2)= d(1, K+2) - d(1, 2) = K-1$, which is not true. We also have a contradiction from the same argument for the case that $L(1, y_{m}, K+2) = K$. 
This completes a proof for the case that $k$ is odd.

Next we prove for the case that $k$ is even. We put $k = 2K$, and let $1, 2, \dots, 2K$ be vertices of $2K$-cycle with $e = \{1, 2\}$ similarly to the odd case. Note that we have $d(1, K+1) = K$. We show that the homology cycle $[(1, 2, K+1) - (1, 2K, K+1)] \in \MH^{1, K+1}_{2,K}(G)$ is non-trivial. Assume that we have $[(1, 2, K+1) - (1, 2K, K+1)] = 0$, that is, there exist tuples $\alpha_{1}, \dots, \alpha_{n} \in \MC^{1, K+1}_{3,K}(G)$ and vertices $a, b\in V(G)$ such that 
\begin{align*}
\partial  &\Big( (1, 2, a, K+1) + (-1)^{s_1}\alpha_{1} +  \dots + (-1)^{s_n}\alpha_{n} - (1, 2K, b, K+1) \Big) \\ &= (1, 2, K+1) - (1, 2K, K+1).
\end{align*}
Note that any tuples of the form $(1, a, 2, K+1)$ and $(1, b, 2K, K+1)$ do not appear in $\alpha_i$'s, because $L(1, a, 2, K+1), L(1, b, 2K, K+1) > K$. We put $\alpha_{0} = (1, 2, a, K+1)$, $\alpha_{n+1} =  (1, 2K, b, K+1)$, and $\alpha_{i} = (1, x_{i}, y_{i}, K+2)$ for $i\in\{1,\ldots,n\}$. Similarly to the odd case, we construct a graph $A(G)$ with vertices 
\[
\{2, a, x_{1}, y_{1},  \dots, x_{n}, y_{n}, 2K, b\}.
\]
Then the same statement in Lemma \ref{2nd} holds. The proof is almost the same as that of Lemma \ref{2nd} as follows.
\begin{proof}[Proof of Lemma \ref{2nd} for $k=2K$ case]
Let $\langle 2, a_{1}, a_{2}, \dots, x = a_{N} \rangle$ be a path in $A(G)$ connecting 2 and $x$. Note that $a_{1}$ satisfies that $d(1, 2) + d(2, a_{1}) + d(a_{1}, K+1) = K$ because $(1, 2, a_{1}, K + 1) = \alpha_{m}$ for some $m$. Let $\langle 1, b^{i}_{1}, \dots, a_{i}\rangle$ be a shortest path in $G$ connecting 1 and $a_i$. We show that $b^{i}_{1} = 2$ by induction on $i$. If $b^{1}_1 \neq 2$, then a closed path obtained by concatenating three paths, $\langle 1, b^{1}_1, \dots, a_1 \rangle$, a shortest path connecting $a_1$ and $2$, and the edge between $2$ and $1$ produces a cycle containing $e$. Note here that the shortest path from $2$ to $a_{1}$ does not pass through $1$. If it goes through $1$, then we have $K = d(1, 2) + d(2, a_{1}) + d(a_{1}, K+1) = 2 + d(1, a_{1}) + d(a_{1}, K + 1) \geq 2 + d(1, K+1) = K+2$. Because $d(1, 2) + d(2, a_1) \leq K-1$, the obtained cycle is of length $\leq 2K-2$, which contradicts the assumption. Hence we have $b^{1}_1 = 2$.

Suppose $b^{i}_{1} = 2$ and $b^{i+1}_{1} \neq 2$. If $(1, a_{i},  a_{i+1}, K+1) = \alpha_{m}$ for some $m$, then a closed path obtained by concatenating three paths, $\langle1, b^{i}_1, \dots, a_i\rangle$, a shortest path connecting $a_i$ and $a_{i+1}$, and $\langle a_{i+1}, \dots, b^{i+1}_1, 1\rangle$ produces a cycle containing $e$.
Note here that the shortest path from $a_i$ to $a_{i+1}$ does not pass through $1$ in the same manner as discussed above.
Because $d(1, a_{i}) + d(a_{i}, a_{i+1}) \leq K-1$, the obtained cycle is of length $\leq 2K-2$, which contradicts the assumption. Similarly, if $(1, a_{i+1},  a_{i}, K+2) = \alpha_{m}$ for some $m$, then a closed path obtained by concatenating three paths, $\langle 1, b^{i}_1, \dots, a_i \rangle $, a shortest path connecting $a_i$ and $a_{i+1}$, and $\langle a_{i+1}, \dots, b^{i+1}_1, 1 \rangle$ produces a cycle containing $e$. Because $d(1, a_{i+1}) + d(a_{i+1}, a_{i}) \leq K-1$, the obtained cycle is of length $\leq 2K-2$, which also contradicts the assumption. Hence we have $b^{i+1}_1 = 2$.
\end{proof}
Now we can show that the vertices $2$ and $b$ in $A(G)$ belong to the same connected component as follows. Divide the collection of tuples 
\[
 (1, 2, a, K+1), \alpha_{1}, \dots, \alpha_{n}, (1, 2K, b, K+1)
 \]
into subcollections $C_{0}, \dots, C_{M}$ corresponding to the connected components of $A(G)$. Suppose that $(1, 2, a, K+1) \in C_{0}$ and $(1, 2K, b, K+1) \in C_{1}$. By the same argument as that in the odd case, we have 
\[
\partial \Big( \sum_{i\geq 2}\sum_{\alpha_{j} \in C_{i}} (-1)^{s_{j}}\alpha_{j} \Big) = 0.
\]
Because $d(1, K+1) = K$, every tuple $\alpha_{i}$ has no singular points other than the end points. Hence two chains $\partial \Big( \sum_{\alpha_{j} \in C_{0}} (-1)^{s_{j}}\alpha_{j} \Big)$ and $\partial \Big( \sum_{\alpha_{j} \in C_{1}} (-1)^{s_{j}}\alpha_{j} \Big)$ must have a common term up to sign. It contradicts the disconnectedness assumption for $C_{0}$ and $C_{1}$, hence the vertices $2$ and $b$ in $A(G)$ belong to the same connected component.
Since the tuple $(1,2K,b,K+1)$ has no singular points, a path in $G$ obtained by concatenating the edge between $1$ and $2K$, and a shortest path connecting $2K$ and $b$ is a shortest path connecting $1$ and $b$. This is a contradiction because every shortest path in $G$ connecting $1$ and $b$ passes through $2$ at the first step by Lemma~\ref{2nd} for $k=2K$ case.
\end{proof}

\section{Stochastic properties of magnitude homology}\label{stochasticpart}
\label{sec:proofs}

\subsection{Phase transition of diagonality}
In this subsection, we provide the proof of Theorem~\ref{thm:main1}. We first prove Theorem~\ref{thm:main1}~(1) which follows from the fact that a.a.s. $G_{n,p}$ has no cycles whenever $p=o(n^{-1})$. In what follows, for $i\ge3$, we denote by $C_i$ the number of $i$-cycles in $G_{n,p}$.
\begin{proof}[Proof of Theorem~\ref{thm:main1}~$(1)$]
For $i\ge3$, a straightforward calculation yields
\[
\E C_i
\le\binom ni\frac{i!}{2i}p^i
\le\frac{(np)^i}{2i}.
\]
Indeed, there are $\binom ni$ ways of selecting $i$ vertices of an $i$-cycle from $n$ vertices, and to each selection, there are $i!/(2i)$ ways of choosing the edges of the $i$-cycle. Lastly, the probability that the chosen $i$ edges are included in $G_{n,p}$ is $p^i$ because of the mutual independence of edge appearance.
As seen in Example~\ref{eg:tree}, all trees, or more generally forests, are diagonal. Therefore, we have
\[
\P(G_{n,p}\text{ is non-diagonal})
\le\P\Biggl(\sum_{i=3}^\infty C_i\ge1\Biggr)
\le\sum_{k=3}^\infty\E C_i
\le\sum_{i=3}^\infty\frac{(np)^i}{2i}.
\]
In the second inequality, we use Markov's inequality. The right-hand side converges to zero as $n\to\infty$, which completes the proof. 
\end{proof}

We now turn to proving Theorem~\ref{thm:main1}~(2)~(3). For their proofs, we divide the concerned regime of $p$ into two parts:
\begin{enumerate}
\item $p=cn^{-1}$ for some $0<c<1$, \item $\liminf_{n\to\infty}np>1$ and $p=o(n^{-3/4})$.
\end{enumerate}
We then discuss the asymptotic behavior of $\P(G_{n,p}\text{ is non-diagonal})$ in each part in different ways.

For the estimate of $\P(G_{n,p}\text{ is non-diagonal})$ in part~(1), we use the following lemma which states that almost all vertices belong to tree components and that there exist no components containing more than one cycle. Let $T(G_{n,p})$ denote the number of vertices in $G_{n,p}$ belonging to some tree component.
\begin{lem}[Theorem~5.7~(ii) and Corollary~5.8 in~\cite{Bo01}]\label{lem:treeUni}
Let $p=cn^{-1}$ for some fixed $0<c<1$. Then, $\E[T(G_{n,p})]=n-O(1)$. In addition, every component is either tree or unicyclic a.a.s.
\end{lem}
The following lemma is also useful.
\begin{lem}[Corollary~4.9 in~\cite{Bo01}]\label{lem:jointDist}
Let $p=cn^{-1}$ for some fixed $c>0$. Then, for any $m\ge3$,
\[
(C_3,C_4,\ldots,C_m)\to(Z_3,Z_4,\ldots,Z_m)\quad\text{in distribution as }n\to\infty,
\]
where $\{Z_i\}_{i=3}^m$ are mutually independent random variables, and each $Z_i$ follows the Poisson distribution with parameter $c^i/(2i)$.
In other words, for any $m\ge3$ and $(a_3,a_4,\ldots,a_m)\in\Z_{\ge0}^{m-2}$,
\[
\lim_{n\to\infty}\P((C_3,C_4,\ldots,C_m)=(a_3,a_4,\ldots,a_m))
=\prod_{i=3}^m\frac{\{c^i/(2i)\}^{a_i}}{a_i!}\exp\biggl(-\frac{c^i}{2i}\biggr).
\]
\end{lem}
Combining Lemmas~\ref{lem:treeUni} and~\ref{lem:jointDist}, we obtain the estimate of $\P(G_{n,p}\text{ is diagonal})$ in part~(1) as follows.
\begin{prop}\label{prop:part1}
Let $p=cn^{-1}$ for some fixed $0<c<1$. Then,
\[
\lim_{n\to\infty}\P(G_{n,p}\text{ is diagonal})=\sqrt{1-c}\exp(c/2+c^2/4+c^3/6+c^4/8).
\]
\end{prop}
\begin{proof}
Let $F_1$ and $F_2$ denote the events that $G_{n,p}$ is diagonal and that $G_{n,p}$ does not contain any cycles of length at least $5$, respectively. We additionally define $E$ as the event that every component in $G_{n,p}$ is either tree or unicyclic.
We can confirm that every unicyclic component that has a cycle of length at least $5$ is non-diagonal. This follows from the Mayer--Vietoris Theorem for magnitude homology~\cite[Theorem~6.6]{HW17} combining with the fact that any cycle graphs of length at least $5$ are non-diagonal (cf.~\cite[Theorems~4.6 and~4.8]{Gu}).
Therefore, we have $E\cap F_1\subset E\cap F_2$. On the other hand, it holds that $E\cap F_1\supset E\cap F_2$ by using again the Mayer--Vietoris Theorem with the fact that tree graphs and $3$- or $4$-cycle graphs are diagonal (cf.~\cite[Examples~2.5 and~5.4]{HW17}).
Consequently, we obtain $E\cap F_1=E\cap F_2$. Thus, it reduces to prove that
\begin{equation}\label{eq:part1}
\lim_{n\to\infty}\P(F_2)
=\sqrt{1-c}\exp(c/2+c^2/4+c^3/6+c^4/8).
\end{equation}
Indeed, $|\P(F_1)-\P(F_2)|=|\P(F_1\setminus E)-\P(F_2\setminus E)|\le\P(E^c)=o(1)$ from the second conclusion of Lemma~\ref{lem:treeUni}.

Now, let $m\ge5$ be fixed, and let $D$ denote the event that every cyclic component has at most $m$ vertices. Then, we have
\begin{align}\label{eq:part1sand}
\P(C_5=C_6=\cdots=C_m=0)
\ge\P(F_2)
&\ge\P(\{C_5=C_6=\cdots=C_m=0\}\cap D)\nonumber\\
&\ge\P(C_5=C_6=\cdots=C_m=0)-\P(D^c).
\end{align}
From the first conclusion of Lemma~\ref{lem:treeUni}, we can take a constant $K$, depending only on $c$, such that $n-\E[T(G_{n,p})]\le K$ for all $n$. Since the number of cyclic components that have more than $m$ vertices is bounded above by $\{n-T(G_{n,p})\}/m$, we obtain
\[
\P(D^c)
\le\frac{n-\E[T(G_{n,p})]}m
\le\frac Km
\]
using Markov's inequality in the first inequality.
Furthermore, Lemma~\ref{lem:jointDist} yields
\[
\lim_{n\to\infty}\P(C_5=C_6=\cdots=C_m=0)
=\prod_{i=5}^m\exp\biggl(-\frac{c^i}{2i}\biggr)
=\exp\Biggl(-\frac12\sum_{i=5}^m\frac{c^i}i\Biggr).
\]
Combining the above estimates with Eq.~\eqref{eq:part1sand}, we obtain
\[
\exp\Biggl(-\frac12\sum_{i=5}^m\frac{c^i}i\Biggr)
\ge\limsup_{n\to\infty}\P(F_2)
\ge\liminf_{n\to\infty}\P(F_2)
\ge\exp\Biggl(-\frac12\sum_{i=5}^m\frac{c^i}i\Biggr)-\frac Km.
\]
Eq.~\eqref{eq:part1} follows from the equation above by taking $m\to\infty$, noting that
\[
\exp\Biggl(-\frac12\sum_{i=5}^\infty\frac{c^i}i\Biggr)
=\sqrt{1-c}\exp(c/2+c^2/4+c^3/6+c^4/8).\qedhere
\]
\end{proof}

For the estimate of $\P(G_{n,p}\text{ is non-diagonal})$ in part~(2), we use the following lemma. For a graph $G$, let us denote the number of connected components of $G$ by $\xi(G)$.
\begin{lem}[{\cite[Section~6]{ER60}}]\label{lem:compDecay}
Let $p=cn^{-1}$ for some fixed constant $c>0$. Then, for any $\eps>0$,
\[
\lim_{n\to\infty}\P\biggl(\biggl|\frac{\xi(G_{n,p})}n-u(c)\biggr|>\eps\biggr)=0,
\]
where
\[
u(c)=\frac1c\sum_{i=1}^\infty\frac{i^{i-2}}{i!}(ce^{-c})^i.
\]
\end{lem}
\begin{figure}[H]
\centering
\includegraphics[width=10cm,bb=0 0 672 258]{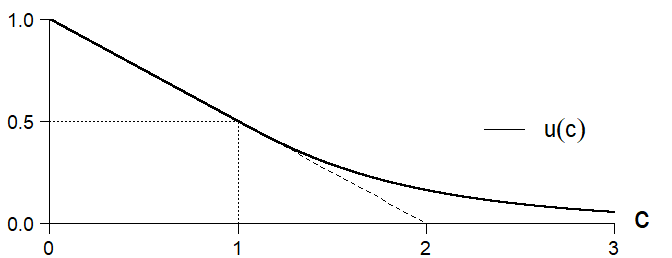}\\
\caption{Description of $u(c)$ in Lemma~\ref{lem:compDecay}.}
\label{fig:compDecay}
\end{figure}
For a graph $G$, the {\it circuit rank} $r(G)$ indicates the minimum number of edges that must be removed from $G$ to contain no cycles. As a well-known fact, it holds that $r(G)=|E(G)|-|V(G)|+\xi(G)$.
\begin{lem}\label{lem:cirRk}
Let $p=cn^{-1}$ for some fixed constant $c>1$. Then, there exists a constant $\dl>0$ such that $r(G_{n,p})\ge\dl n$ a.a.s.
\end{lem}
\begin{proof}
We can verify that $u(c)>1-c/2$ whenever $c>1$~(see also Figure~\ref{fig:compDecay}). Therefore, Lemma~\ref{lem:compDecay} implies that for $c>1$, there exists a constant $\dl>0$ such that $\xi(G_{n,p})\ge(1-c/2+2\dl)n$ a.a.s.
Furthermore, since $\#E(G_{n,p})$ follows the binomial distribution with parameters $\binom n2$ and $cn^{-1}$, a direct computation yields
\begin{align*}
\E\biggl[\frac{\#E(G_{n,p})}n\biggr]=\frac c2\biggl(1-\frac1n\biggr)\xrightarrow[n\to\infty]{}\frac c2
\shortintertext{ and }
\var\biggl(\frac{\#E(G_{n,p})}n\biggr)=\frac c{2n}\biggl(1-\frac1n\biggr)\biggl(1-\frac cn\biggr)\xrightarrow[n\to\infty]{}0.
\end{align*}
Therefore, using the Minkowski inequality,
\begin{align}\label{eq:L2conv}
\E\Biggl[\biggl(\frac{\#E(G_{n,p})}n-\frac c2\biggr)^2\Biggr]^{1/2}
&\le\E\Biggl[\biggl(\frac{\#E(G_{n,p})}n-\E\biggl[\frac{\#E(G_{n,p})}n\biggr]\biggr)^2\Biggr]^{1/2}+
\biggl|\E\biggl[\frac{\#E(G_{n,p})}n\biggr]-\frac c2\biggr|\nonumber\\
&=\sqrt{\var\biggl(\frac{\#E(G_{n,p})}n\biggr)}+\biggl|\E\biggl[\frac{\#E(G_{n,p})}n\biggr]-\frac c2\biggr|
\xrightarrow[n\to\infty]{}0.
\end{align}
Thus, from Markov's inequality, we have $\#E(G_{n,p})\ge(c/2-\dl)n$ a.a.s. Combining these estimates above, we obtain a.a.s. $r(G_{n,p})=\#E(G_{n,p})-n+\xi(G_{n,p})
\ge(c/2-\dl)n-n+(1-c/2+2\dl)n=\dl n$.
\end{proof}
We now provide the estimate of $\P(G_{n,p}\text{ is non-diagonal})$ in part~(2).
\begin{prop}\label{prop:part2}
Let $\liminf_{n\to\infty}np>1$ and $p=o(n^{-3/4})$. Then, $G_{n,p}$ is non-diagonal a.a.s.
\end{prop}
\begin{proof}
Let $X$ denote the number of edges $e\in E(G_{n,p})$ such that $\gir_e(G_{n,p})\in[5,\infty)$. From Theorem~\ref{nondiag}, it suffices to prove that $X\ge1$ a.a.s.
We define $Y$ as the number of edges that are contained in some cycle. Then, $Y\ge r(G_{n,p})$ because of the definition of the circuit rank. Thus, by applying Lemma~\ref{lem:cirRk} with some fixed constant $1<c<\liminf_{n\to\infty}np$, there exists a constant $\dl>0$ such that $Y\ge r(G_{n,p})\ge\dl n$ a.a.s.
For $i\ge3$, we additionally define $Y_i$ as the number of edges that are contained in some $i$-cycle. Then,
\[
\P\biggl(Y_i>\frac\dl 3n\biggr)
\le\frac3{\dl n}\E Y_i
\le\frac{3i}{\dl n}\E C_i
\le\frac{3i}{\dl n}\frac{(np)^i}{2i}
=\frac3{2\dl}n^{i-1}p^i.
\]
The first inequality follows from Markov's inequality. In the second inequality, we use a crude estimate $Y_i\le iC_i$. Since $p=o(n^{-3/4})$, for $i=3,4$, the right-hand side of the above equation converges to zero as $n\to\infty$. Therefore, $Y_3,Y_4\le\dl n/3$ a.a.s.
Combining the estimates for $Y$, $Y_3$, and $Y_4$,
\[
\P\biggl(X\ge\frac\dl 3n\biggr)
\ge\P\biggl(Y-Y_3-Y_4\ge\frac\dl 3n\biggr)
\ge\P\biggl(Y\ge\dl n\text{ and }Y_3,Y_4\le\frac\dl 3n\biggr)\xrightarrow[n\to\infty]{}1,
\]
which completes the proof.
\end{proof}
Combining Propositions~\ref{prop:part1} and~\ref{prop:part2}, we obtain the conclusion of Theorem~\ref{thm:main1}.

Lastly, we prove Theorem~\ref{thm:main2}. The notion of pawful graphs, introduced  by Gu~\cite{Gu}, is a key for the proof. Recall from Definition~\ref{df:pawful} that a pawful graph $G$ is a graph of diameter at most two satisfying the property that for any distinct vertices $x,y,z\in V(G)$ with $d(x,y)=d(y,z)=2$ and $d(z,x)=1$, they have a common neighbor.
Since pawful graphs are diagonal, the conclusion of Theorem~\ref{thm:main2} follows immediately from the following Theorem.
\begin{thm}[{\cite[Theorem~3.2]{Ka09}}]
Let $m\in\N$ and $\eps>0$. Then,
\[
p\ge\biggl(\frac{(m+\eps)\log n}n\biggr)^{1/m}
\]
implies that every $m$ vertices in $G_{n,p}$ have a common neighbor a.a.s.
\end{thm}

\subsection{Weak law of large numbers for the rank of magnitude homology}
In this subsection, we prove Theorem~\ref{thm:main3} using Theorem~\ref{otherpart}. We first give a general upper bound of the rank of magnitude homology of a graph.
\begin{lem}\label{lem:rkBound}
Let $G$ be a graph, and let $x\in V(G)$ be fixed. Then, for any $k,\ell\in\N$,
\[
\rk(\MH_{k,\ell}^x(G))\le\binom{\ell-1}{k-1}\Bigl(\max_{y\in V(G)}\deg y\Bigr)^\ell.
\]
\end{lem}
\begin{proof}
Recall that the generator set of $\MC_{k,\ell}^x(G)$ is
\begin{align*}
&\left\{(x_0,x_1,\ldots,x_k)\in V(G)^{k+1}\relmiddle|x=x_0\neq x_1\neq\cdots\neq x_k,\sum_{i=1}^kd(x_{i-1},x_i)=\ell\right\}\\
&=\bigsqcup_{\substack{(\ell_1,\ell_2,\ldots,\ell_k)\in\N^k\\\ell_1+\ell_2+\cdots+\ell_k=\ell}}\{(x_0,x_1,\ldots,x_k)\in V(G)^{k+1}\mid x_0=x,d(x_{i-1},x_i)=\ell_i\text{ for }1\le i\le k\}.
\end{align*}
Noting that for any $u\in V(G)$ and $r\in\N$,
\[
\#\{v\in V(G)\mid d(u,v)=r\}\le\Bigl(\max_{y\in V(G)}\deg y\Bigr)^r,
\]
we have
\begin{align*}
&\#\{(x_0,x_1,\ldots,x_k)\in V(G)^{k+1}\mid x_0=x,d(x_{i-1},x_i)=\ell_i\text{ for }1\le i\le k\}\\
&\le\prod_{i=1}^k\Bigl(\max_{y\in V(G)}\deg y\Bigr)^{\ell_i}\\
&=\Bigl(\max_{y\in V(G)}\deg y\Bigr)^\ell
\end{align*}
for any $(\ell_1,\ell_2,\ldots,\ell_k)\in\N^k$ with $\ell_1+\ell_2+\cdots+\ell_k=\ell$.
Furthermore, a simple combinatorial argument yields
\[
\#\{(\ell_1,\ell_2,\ldots,\ell_k)\in\N^k \mid \ell_1+\ell_2+\cdots+\ell_k=\ell\}=\binom{\ell-1}{k-1}.
\]
Thus, we conclude that
\[
\rk(\MH_{k,\ell}^x(G))\le\rk(\MC_{k,\ell}^x(G))\le\binom{\ell-1}{k-1}\Bigl(\max_{y\in V(G)}\deg y\Bigr)^\ell.\qedhere
\]
\end{proof}
The following lemma gives a useful upper bounds of the probability that a binomial distributed random variable is larger than expected.
\begin{lem}[{\cite[Lemma~1.1]{Pe03}}]\label{lem:Chernoff}
Suppose $N\in\N$, $p\in(0,1)$, and $0<k<N$. Let $X$ be a binomial random variable with parameters $N$ and $p$, and set $\mu\coloneqq\E X=Np$. If $k\ge e^2\mu$, then
\[
\P(X>k)\le\exp\biggl(-\frac k2\log\biggl(\frac k\mu\biggr)\biggr).
\]
\end{lem}
In what follows, let the Erd\H os--R\'enyi graph $G_{n,p}$ be constructed on an $n$-vertex set $V_n$, and let $o\in V_n$ be an arbitrarily fixed vertex.
\begin{lem}\label{lem:llrkBound}
Let $k,\ell\in\N$ be fixed. It holds that for sufficiently large $n$ and any $x\in V_n$,
\[
\E[\rk(\MH_{k,\ell}^x(G_{n,p}))^2]\le\binom{\ell-1}{k-1}^2(\log n)^{2\ell}.
\]
\end{lem}
\begin{proof}
Let $D$ be the event that the maximum degree of $G_{n,p}$ is at most $(\log n)/2$. Then,
\[
\P(D^c)
\le\sum_{y\in V_n}\P\Bigl(\deg y>\frac{\log n}2\Bigr)
=n\P\Bigl(\deg o>\frac{\log n}2\Bigr).
\]
Note that $\deg o$ follows the binomial distribution with parameters $n-1$ and $cn^{-1}$, and set $\mu\coloneqq\E[\deg o]=(n-1)cn^{-1}$. Applying Lemma~\ref{lem:Chernoff} with $N=n-1$, $p=cn^{-1}$, and $k=(\log n)/2$, we have
\[
\P\Bigl(\deg o>\frac{\log n}2\Bigr)\le\exp\Bigl(-\frac{\log n}4\log\Bigl(\frac{\log n}{2\mu}\Bigr)\Bigr)
\le\exp\Bigl(-\frac15\log n\log\log n\Bigr)=n^{-(\log\log n)/5}
\]
for sufficiently large $n$. Therefore, for sufficiently large $n$ and any $x\in V_n$, we obtain
\begin{align*}
\E\bigl[\rk(\MH_{k,\ell}^x(G))^2\bigr]
&\le\binom{\ell-1}{k-1}^2\E\biggl[\Bigl(\max_{y\in V_n}\deg y\Bigr)^{2\ell}\biggr]\\
&\le\binom{\ell-1}{k-1}^2\biggl\{\E\biggl[\Bigl(\max_{y\in V_n}\deg y\Bigr)^{2\ell};D\biggr]+n^{2\ell}\P(D^c)\biggr\}\\
&\le\binom{\ell-1}{k-1}^2\biggl\{\Bigl(\frac{\log n}2\Bigr)^{2\ell}+n^{2\ell+1-(\log\log n)/5}\biggr\}\\
&\le\binom{\ell-1}{k-1}^2(\log n)^{2\ell}.
\end{align*}
In the first inequality, we use Lemma~\ref{lem:rkBound}.
\end{proof}

We now trun to proving Theorem~\ref{thm:main3} using Theorem~\ref{otherpart}.
\begin{proof}[Proof of Theorem~\ref{thm:main3}]
Since $\MH_{k,\ell}(G_{n,p})=0$ if $\ell<k$, we assume that $\ell\ge k$. For $i\ge3$, define $E_i^x$ as the event that $G_{n,p}$ has at least one $i$-cycle containing $x$, and set
\[
E^x\coloneqq\bigcup_{i=3}^{2(\ell-k)+4}E_i^x.
\]
Applying Theorem~\ref{otherpart}, we have
\begin{align*}
\frac{\rk(\MH_{k,\ell}(G_{n,p}))}n
&=\frac1n\sum_{x\in V_n}\rk(\MH_{k,\ell}^x(G_{n,p}))\\
&\le\frac1n\sum_{x\in V_n}\bigl\{(\deg x)\dl_{k,\ell}+\rk(\MH_{k,\ell}^x(G_{n,p}))1_{E^x}\bigr\}\\
&=\frac{2\#E(G_{n,p})}n\dl_{k,\ell}+\frac1n\sum_{x\in V_n}\rk(\MH_{k,\ell}^x(G_{n,p}))1_{E^x}.
\end{align*}
On the other hand, since $\rk(\MH_{\ell,\ell}(G_{n,p}))\ge2\#E(G_{n,p})$,
we have
\[
\frac{\rk(\MH_{k,\ell}(G_{n,p}))}n
\ge\frac{2\#E(G_{n,p})}n\dl_{k,\ell}.
\]
Combining these estimates, we obtain
\[
\biggl|\frac{\rk(\MH_{k,\ell}(G_{n,p}))}n-\frac{2\#E(G_{n,p})}n\dl_{k,\ell}\biggr|
\le\frac1n\sum_{x\in V_n}\rk(\MH_{k,\ell}^x(G_{n,p}))1_{E^x}.
\]
Therefore, using the triangle inequality,
\begin{align}\label{eq:L1norm}
&\E\biggl|\frac{\rk(\MH_{k,\ell}(G_{n,p}))}n-c\dl_{k,\ell}\biggr|\nonumber\\
&\le\E\biggl|\frac{\rk(\MH_{k,\ell}(G_{n,p}))}n-\frac{2\#E(G_{n,p})}n\dl_{k,\ell}\biggr|+\E\biggr|\frac{2\#E(G_{n,p})}n\dl_{k,\ell}-c\dl_{k,\ell}\biggr|\nonumber\\
&\le\frac1n\sum_{x\in V_n}\E[\rk(\MH_{k,\ell}^x(G_{n,p}))1_{E^x}]+\E\biggr|\frac{2\#E(G_{n,p})}n-c\biggr|\dl_{k,\ell}\nonumber\\
&\le\E[\rk(\MH_{k,\ell}^o(G_{n,p}))1_{E^o}]+\E\biggr|\frac{2\#E(G_{n,p})}n-c\biggr|\nonumber\\
&\le\E\bigl[\rk(\MH_{k,\ell}^o(G_{n,p}))^2\bigr]^{1/2}\P\bigl(E^o\bigr)^{1/2}+\E\Biggr[\biggr(\frac{2\#E(G_{n,p})}n-c\biggr)^2\Biggr]^{1/2}.
\end{align}
In the last line, we use the Cauchy--Schwarz inequality.
The second term of Eq.~\eqref{eq:L1norm} converges to zero as $n\to\infty$, as seen in Eq.~\eqref{eq:L2conv}.
For the estimate of the first term in Eq.~\eqref{eq:L1norm}, we define $C_i^o$ as the number of $i$-cycles containing $o$. We then have
\[
\P(E_i^o)
=\P(C_i^o\ge1)
\le\E C_i^o
=\frac{(n-1)(n-2)\cdots(n-i+1)}2\Bigl(\frac cn\Bigr)^i
\le\frac{c^i}{2n}
\]
from Markov's inequality, which implies that
\[
\P(E^o)
\le\sum_{i=3}^{2(\ell-k)+4}\P(E_i^o)
\le\frac1{2n}\sum_{i=3}^{2(\ell-k)+4}c^i.
\]
From the estimate above and Lemma~\ref{lem:llrkBound}, the first term of Eq.~\eqref{eq:L1norm} converges to zero as $n\to\infty$.
Consequently, we obtain
\[
\lim_{n\to\infty}\E\biggl|\frac{\rk(\MH_{k,\ell}(G_{n,p}))}n-c\dl_{k,\ell}\biggr|=0,
\]
which implies the first conclusion. Again from Markov's inequality, the above equation also implies the second conclusion.
\end{proof}



\end{document}